\documentclass[a4paper,10pt]{article}

 \usepackage{amsfonts}
 \usepackage[ansinew]{inputenc}
 \usepackage{amsmath}
 \usepackage{amssymb}
 \usepackage{amsthm}
 \usepackage[arrow, matrix, curve]{xy}
 \usepackage{tikz}
 \usepackage[numbers]{natbib}
 \usetikzlibrary{arrows}
 \usetikzlibrary{snakes}
 \usetikzlibrary{shapes}

\newtheorem{theorem}{Theorem}[section]
\newtheorem{lemma}[theorem]{Lemma}
\newtheorem{prop}[theorem]{Proposition}
\newtheorem{defn}[theorem]{Definition}
\newtheorem{cor}[theorem]{Corollary}

\title{A quantum cluster algebra of Kronecker type and the dual canonical basis}
\author{Philipp Lampe}

\begin{document}

\maketitle
\begin{abstract}
The article concerns the dual of Lusztig's canonical basis of a subalgebra of the positive part $U_q(\mathfrak{n})$ of the universal enveloping algebra of a Kac-Moody Lie algebra of type $A_1^{(1)}$. The examined subalgebra is associated with a terminal module $M$ over the path algebra of the Kronecker quiver via an Weyl group element $w$ of length four. 

Gei\ss -Leclerc-Schr\"{o}er attached to $M$ a category $\mathcal{C}_M$ of nilpotent modules over the preprojective algebra of the Kronecker quiver together with an acyclic cluster algebra $\mathcal{A}(\mathcal{C}_M)$. The dual semicanonical basis contains all cluster monomials. By construction, the cluster algebra $\mathcal{A}(\mathcal{C}_M)$ is a subalgebra of the graded dual of the (non-quantized) universal enveloping algebra $U(\mathfrak{n})$.

We transfer to the quantized setup. Following Lusztig we attach to $w$ a subalgebra $U_q^+(w)$ of $U_q(\mathfrak{n})$. The subalgebra is generated by four elements that satisfy straightening relations; it degenerates to a commutative algebra in the classical limit $q=1$.  The algebra $U_q^+(w)$ possesses four bases, a PBW basis, a canonical basis, and their duals. We prove recursions for dual canonical basis elements. The recursions imply that every cluster variable in $\mathcal{A}(\mathcal{C}_M)$ is the specialization of the dual of an appropriate canonical basis element. Therefore, $U_q^+(w)$ is a quantum cluster algebra in the sense of Berenstein-Zelevinsky. Furthermore, we give explicit formulae for the quantized cluster variables and for expansions of products of dual canonical basis elements. 
\end{abstract}

\section{Introduction}
Cluster algebras have been introduced by Fomin and Zelevinsky (see \cite{FZ}, \cite{FZ2}) as a device for studying {\it dual canonical bases} and {\it total positivity}. 

A {\it cluster algebra} of rank $n$ (for some $n \in \mathbb{N}$) is a subalgebra of the field $\mathbb{Q}(x_1,\ldots,x_n)$ of rational functions in $n$ variables. Its generators are called {\it cluster variables}. Each cluster variable belongs to several overlapping {\it clusters}. Every cluster, and hence every cluster variable, is obtained from an initial cluster by a sequence of {\it mutations}. A mutation replaces an element in a cluster by an explicitly defined rational function in the variables of that cluster. We refer to \cite{FZ} for definitions and to \cite{FZ4} for good survey about cluster algebras. Note that a cluster algebra is {\it commutative}. One feature of cluster algebras is the {\it Laurent phenomenon} asserting that every cluster variable is a Laurent polynomial in the initial cluster variables. 

Cluster algebras quickly gained popularity in various branches of mathematics, for example Poisson geometry (e.g. \cite{GSV}), higher Teichm\"{u}ller theory (e.g. \cite{FG}), combinatorics (e.g. \cite{MP}), integrable systems (e.g. \cite{FZ5}), etc. 

In this article we come back to one of the original motivations for introducing cluster algebras, namely a combinatorial approach to the study of {\it dual canonical bases} of quantized universal enveloping algebras.  Canonical bases have been defined by Lusztig in \cite{Lu2} and \cite{Lu1}. Lusztig uses the algebraic geometry of varieties of representations of a quiver and especially the theory of perverse sheaves to construct canonical bases. The dual canonical basis is the basis adjoint to the canonical basis with respect to a bilinear form. Lusztig (see \cite[Theorem 14.2.3]{Lu1}) also showed that there is a combinatorial characterization for the canonical basis elements. The corresponding version for dual canonical bases is described in \cite{lnt}, \cite{Re}, and \cite{Le2}; in our example the characterization is Theorem \ref{combprop}. Fomin and Zelevinsky (see \cite{FZ}) conjecture that every cluster monomial (i.e. a monomial in the cluster variables of a single cluster) is a non-deformed dual canonical basis element. The conjecture has only been proven in very few cases. 

Let $Q$ be a finite quiver without oriented cycles. Furthermore, let $\mathbb{C}Q$ be its {\it path algebra} and let $\Lambda$ be the associated {\it preprojective algebra}. Gei\ss -Leclerc-Schr\"oer (\cite[Section 4]{GLS2}) attached to every {\it terminal} $\mathbb{C}Q$-module $M$ a subcategory $\mathcal{C}_M \subseteq {\rm nil}(\Lambda)$ of nilpotent $\Lambda$-modules. The category $\mathcal{C}_M$ is a {\it Frobenius category}. The stable category $\underline{\mathcal{C}}_M$ is triangulated by a theorem of Happel (\cite[Section 2.6]{H}). Gei\ss -Leclerc-Schr\"oer (\cite[Theorem 11.1]{GLS2}) showed that if $M=I \oplus \tau(I)$ where $I$ is the direct sum of all indecomposable injective representations, then there is an equivalence of triangulated categories $\underline{\mathcal{C}}_M \simeq \mathcal{C}_Q$ between $\underline{\mathcal{C}}_M$ and the {\it cluster category} $\mathcal{C}_Q$ as defined in \cite{BMRRT} to be the orbit category $\mathcal{D}^b\left({\rm mod}(kQ)\right)/\tau_{\mathcal{D}}^{-1} \circ [1]$. Keller (\cite{K}) proved that $\mathcal{C}_Q$ is indeed triangulated. 

Furthermore, Gei\ss -Leclerc-Schr\"oer (\cite[Section 4]{GLS2}) constructed for every $\mathcal{C}_M$ a cluster algebra $\mathcal{A}(\mathcal{C}_M)$; it is a subalgebra of the graded dual of the universal enveloping algebra of the positive part of the corresponding Lie algebra, i.e. $\mathcal{A}(\mathcal{C}_M) \subseteq U(\mathfrak{n})_{gr}^{\ast}$. \cite{GLS2} also proved that all cluster monomials are in the dual of Lusztig's semicanonical basis. There is an isomorphism between $U(\mathfrak{n})$ and an algebra $\mathcal{M}$ of $\mathbb{C}$-valued functions on $\Lambda$. We refer to \cite{GLS2} for a precise definition of $\mathcal{M}$. It is generated by functions $d_{\textbf{i}}$ that map a $\Lambda$-module $X$ to the Euler characteristic of the flag variety of $X$ of type $\textbf{i}$. Prominent elements in $\mathcal{A}(\mathcal{C}_M)$ are (under the described isomorphism) the $\delta$-functions of certain rigid $\Lambda$-modules. Additionally, there is an isomorphism $\mathcal{A}(\mathcal{C}_M) \simeq \mathbb{C}[N(w)]$ where $\mathbb{C}[N(w)]$ is the coordinate ring of the unipotent subgroup $N(w)$ attached to the adaptable Weyl group element $w$ of $M$. Therefore, we may call the $\mathcal{C}_M$ a {\it categorification} of the cluster algebra $\mathcal{A}(\mathcal{C}_M)$. 

Our aim is an investigation of the quantized setup for a specific quiver. We consider the special case where $Q$ is the {\it Kronecker quiver}. The type of the corresponding Lie algebra is $A_1^{(1)}$. As a terminal $\mathbb{C}Q$-module $M$ we choose the direct sum of the two indecomposable injective modules and their Auslander-Reiten translates. We describe the cluster algebra $\mathcal{A}(\mathcal{C}_M)$ in Section \ref{sec:int}; it is a cluster algebra of rank two and it is the simplest example of a cluster algebra with infinitely many cluster variables. The adaptable Weyl group element associated with $M$ is $w=s_1s_2s_1s_2$. In Section \ref{quant} and Section \ref{sec:sub} we study the subalgebra $U_q^{+}(w) \subseteq U_q(\mathfrak{n})$ using Lusztig's $T$-automorphisms of $U_q(\mathfrak{g})$. According to  Lusztig (\cite{Lu1}) there is a canonical basis of $U_q^{+}(w)$; it has been studied by Leclerc (\cite{Le1}, \cite{Le2}). The purpose of this article is to describe the dual canonical basis of $U_q^{+}(w)$ more detailed. The main results are the following. Theorem \ref{thm:clu} asserts that every cluster variable of $\mathcal{A}(\mathcal{C}_M)$ is the specialization of an appropriate dual canonical basis element. The main ingredient for the proof is Theorem \ref{thm:recursion} which provides recursions for dual canonical basis element so that they can be computed explicitly. We also give formulae for these dual canonical basis elements. The formulae are quantized versions of the known formulae for the coefficients of cluster variables in type $A_1^{(1)}$ from \cite{CZ}. Furthermore, we show that two adjacent quantized cluster variables quasi-commute, i.e. they are commutative up to a power of the deformation parameter $q$. This is together with the product expansions of Corollary \ref{cor} part of the study of the multiplicative behaviour of dual canonical bases. Furthermore, we prove a quantum exchange relation for the quantized cluster variables so that $U_q^+(w)$ becomes a {\it quantum cluster algebra} in the sense of Berenstein-Zelevinsky (\cite{BZ2}).

{\bf Acknowledgements.} The article is part of my Ph.D. studies at the University of Bonn. I would like to thank the Bonn International Graduate School in Mathematics (BIGS) for their support. I would like to thank my advisor Jan Schr\"oer for continuous encouragement. I am also grateful to Bernard Leclerc for valuable discussion and explanations.

\section{The cluster algebra structure associated with a terminal $kQ$-module}
\label{sec:int}
\subsection{Representation theory of the Kronecker quiver}
The quiver $Q=(Q_0,Q_1)$ with vertex set $Q_0=\left\lbrace 0,1 \right\rbrace$ and arrow set $Q_1=\left\lbrace a_1,a_2\right\rbrace$ with $a_1,a_2 \colon 0 \to 1$ is called the {\it Kronecker quiver}. 
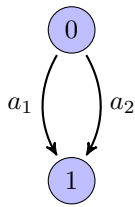
\begin{figure}[h!]
\begin{center}
\begin{tikzpicture}
\node[draw,circle,fill=blue!25] at (8,2) (k1) {0};
\node[draw,circle,fill=blue!25] at (8,0) (k2) {1};
\path[->,thick,shorten <=2pt,shorten >=2pt,>=stealth'] (k1) edge [bend left] node[right] {$a_2$} (k2);
\path[->,thick,shorten <=2pt,shorten >=2pt,>=stealth'] (k1) edge [bend right] node[left] {$a_1$} (k2);
\end{tikzpicture}
\end{center}
\caption{The Kronecker quiver}\label{fig:kron}
\end{figure}

Let $k$ be a field. The category ${\rm rep}_k(Q)$ of finite-dimensional representations of $Q$ can be identified with the category ${\rm mod}(kQ)$ of finite-dimensional modules over the {\it path algebra} $kQ$. (For more information on representations of quivers see for example \cite{CB}.)

The Kronecker quiver is a  {\it tame} quiver. There are infinitely many indecomposable $kQ$-modules which are classified as {\it preprojective}, {\it preinjective} or {\it regular}. A part of the preinjective component of the {\it Auslander-Reiten quiver} of ${\rm mod}(kQ)$ is shown in Figure \ref{fig:AR}. The modules are represented by their dimension vectors. For example, the dimension vector
\begin{center}  
\begin{tikzpicture}
\node[draw,rectangle,rounded corners,fill=blue!25] at (-7,0) (c) {$\begin{matrix}2 \\ 1 \end{matrix}$};
\node at (-4,0) (c) {admits a representation};
\node at (1.5,-0.1) (c) {.};
\node at (0,0.75) (k1) {$k^2$};
\node at (0,-0.75) (k2) {$k$};
\path[->,thick] (k1) edge [bend left] node[right] {$(0 \ 1)$} (k2);
\path[->,thick] (k1) edge [bend right] node[left] {$(1 \ 0)$} (k2);
\end{tikzpicture}
\end{center} The maps are given by $1 \times 2$ matrices if we choose a basis of the vector spaces. The solid arrows display the space of {\it irreducible maps}; the dotted arrows display the {\it Auslander-Reiten translation} $\tau$. 

\begin{figure}[h!]
\begin{center}
\begin{tikzpicture}
\node at (-5,1) (inf) {{\large $\cdots$}};
\node[draw,rectangle,rounded corners,fill=blue!25] at (0,2) (a) {$\begin{matrix}4 \\ 3 \end{matrix}$};
\node[draw,rectangle,rounded corners,fill=blue!25] at (2,0) (b) {$\begin{matrix}3 \\ 2 \end{matrix}$};
\node[draw,rectangle,rounded corners,fill=blue!25] at (4,2) (c) {$\begin{matrix}2 \\ 1 \end{matrix}$};
\node[draw,rectangle,rounded corners,fill=blue!25] at (6,0) (d) {$\begin{matrix}1 \\ 0 \end{matrix}$};
\node[draw,rectangle,rounded corners,fill=red!50] at (-2,0) (e) {$\begin{matrix}5 \\ 4 \end{matrix}$};
\node[draw,rectangle,rounded corners,fill=red!50] at (-4,2) (f) {$\begin{matrix}6 \\ 5 \end{matrix}$};
\path[->,thick,shorten <=2pt,shorten >=2pt] (a) edge [bend left=10] node {} (b);
\path[->,thick,shorten <=2pt,shorten >=2pt] (a) edge [bend left=-10] node {} (b);
\path[->,thick,shorten <=2pt,shorten >=2pt] (b) edge [bend left=10] node {} (c);
\path[->,thick,shorten <=2pt,shorten >=2pt] (b) edge [bend left=-10] node {} (c);
\path[->,thick,shorten <=2pt,shorten >=2pt] (c) edge [bend left=10] node {} (d);
\path[->,thick,shorten <=2pt,shorten >=2pt] (c) edge [bend left=-10] node {} (d);
\path[->,thick,shorten <=2pt,shorten >=2pt] (e) edge [bend left=10] node {} (a);
\path[->,thick,shorten <=2pt,shorten >=2pt] (e) edge [bend left=-10] node {} (a);
\path[->,thick,shorten <=2pt,shorten >=2pt] (f) edge [bend left=-10] node {} (e);
\path[->,thick,shorten <=2pt,shorten >=2pt] (f) edge [bend left=10] node {} (e);
\path[->,thick,shorten <=2pt,shorten >=2pt,dotted] (c) edge node {} (a);
\path[->,thick,shorten <=2pt,shorten >=2pt,dotted] (d) edge node {} (b);
\path[->,thick,shorten <=2pt,shorten >=2pt,dotted] (b) edge node {} (e);
\path[->,thick,shorten <=2pt,shorten >=2pt,dotted] (a) edge node {} (f);
\end{tikzpicture}
\end{center}
\caption{A part of the preinjective component of ${\rm mod}(kQ)$}\label{fig:AR}
\end{figure}
We consider the direct sum $M=I_0 \oplus I_1 \oplus \tau (I_0) \oplus \tau (I_1)$ of the four blue modules. These four modules are the two indecomposable injective modules $I_0$ and $I_1$ asssociated with the vertices $0$ and $1$ and their Auslander-Reiten translates, $\tau (I_0)$ and $\tau (I_1)$. The module $M$ is a {\it terminal $kQ$-module} in the sense of Gei\ss -Leclerc-Schr\"{o}er (\cite[Section 2.2]{GLS2}). According to \cite[Theorem 3.3]{GLS2} the terminal $kQ$-module $M$ gives rise to a cluster algebra structure. To explain this theorem let us introduce some notation.

\subsection{The preprojective algebra} Let $\Lambda$ be the {\it preprojective algebra}; it is defined as $$\Lambda=k\overline{Q} / (c).$$ Here $\overline{Q}$ denotes {\it double quiver} of $Q$ which is by definition given by a vertex set $\overline{Q}_0=Q_0$ and an arrow set $\overline{Q}_1=Q_1 \cup \left\lbrace a_1^{*},a_2^{*}\right\rbrace$ with two additional arrows $a_1^{*},a_2^{*} \colon 1 \to 0$. The ideal $(c)$ is the two-sided ideal generated by the element $$c=\sum_{a \in Q_1}\left(a^{*}a-aa^{*} \right) \in k\overline{Q}.$$ The algebra $\Lambda$ is infinite dimensional, since $Q$ is not an orientation of a Dynkin diagram. There is a {\it restriction functor} $$\pi_Q \colon {\rm mod}(\Lambda) \to {\rm mod}(kQ)$$ given by forgetting the linear maps associated with $a_1^*$ and $a_2^*$ in a $\Lambda$-module i.e. a representation of $\overline{Q}$ such that the linear maps satisfy relations corresponding to the ideal $(c)$. Ringel (\cite[Theorem B]{R}) proved that the category ${\rm mod}(\Lambda)$ is isomorphic to a category called $C(1,\tau)$. The objects in the category $C(1,\tau)$ are pairs $(X,f)$ consisting of a $kQ$-module $X$ and a $kQ$-module homomorphism $f \colon X \to \tau (X)$ from $X$ to its translate $\tau (X)$; morphisms in $C(1,\tau)$ from a pair $(X,f)$ to a pair $(Y,g)$ are given by a $kQ$-module homomorphism $h \colon X \to Y$ for which the diagram 
\begin{center}
\hspace{0cm}
\begin{xy}
  \xymatrix{
    X \ar[r]^{h} \ar[d]^{f} &  Y  \ar[d]^{g}  \\
    \tau (X) \ar[r]^{\tau(h)} &  \tau (Y)  
  }
\end{xy}
\end{center}
commutes. Using this correspondence Gei\ss -Leclerc-Schr\"{o}er (\cite[Section 7.1]{GLS2}) constructed for every $i=0,1$ and natural numbers $a \leq b$ a $\Lambda$-module 
\begin{align*}
T_{i,[a,b]}=\left(I_{i,[a,b]},e_{i,[a,b]} \right) 
\end{align*}
where $I_{i,[a,b]}=\bigoplus_{j=a}^b \tau^j(I_i)$ and the map $$e_{i,[a,b]} \colon I_{i,[a,b]} \to \tau \left( I_{i,[a,b]}\right)=\bigoplus_{j=a+1}^{b+1} \tau^j(I_i) $$ is identity on every $\tau^j(I_i)$ for $a+1 \leq j \leq b$ and zero otherwise. We are interested in the six $\Lambda$-modules $T_{i,[a,b]}$ for $i=0,1$ and $0 \leq a,b \leq 1$. We display the modules by their graded dimension vectors.
\begin{figure}
\begin{center}
\begin{tikzpicture}
\node[rectangle,rounded corners] at (0,0) (a) {$\begin{matrix}0\end{matrix}$};
\node [left] at (a.west) {$T_{0,[0,0]}=$};
\node[rectangle,rounded corners] at (6,0) (b) {$\begin{matrix}0&&0 \\ &1& \end{matrix}$};
\node [left] at (b.west) {$T_{1,[0,0]}=$};
\node[rectangle,rounded corners] at (0,-1.5) (c) {$\begin{matrix}0&&0&&0 \\ &1&&1& \end{matrix}$};
\node [left] at (c.west) {$T_{0,[1,1]}=$};
\node[rectangle,rounded corners] at (6,-1.5) (d) {$\begin{matrix}0&&0&&0&&0 \\ &1&&1&&1& \end{matrix}$};
\node [left] at (d.west) {$T_{1,[1,1]}=$};
\node[rectangle,rounded corners] at (0,-3.5) (e) {$\begin{matrix}0&&0&&0 \\ &1&&1& \\ &&0&&\end{matrix}$};
\node [left] at (e.west) {$T_{0,[0,1]}=$};
\node[rectangle,rounded corners] at (6,-3.5) (f) {$\begin{matrix}0&&0&&0&&0 \\ &1&&1&&1& \\ &&0&&0&& \\ &&&1&&& \end{matrix}$};
\node [left] at (f.west) {$T_{1,[0,1]}=$};
\end{tikzpicture}
\end{center}
\caption{The modules $T_{i,[a,b]}$}\label{tmod}
\end{figure}

All six modules are {\it rigid} and {\it nilpotent}. Recall that a $\Lambda$-module $T$ is said to be rigid if ${\rm Ext}_{\Lambda}^1(T,T)=0$.

\subsection{The $\delta$-functions and the cluster algebra structure}
In this subsection let $k=\mathbb{C}$. The $\delta$-functions of the modules $T_{i,[a,b]}$ satisfy {\it generalized determinantal identities} (see \cite[Theorem 18.1]{GLS2}). Let $U_0,U_1,U_2$ and $U_3$ be the $\delta$-functions of $T_{0,[0,0]},T_{1,[0,0]},T_{0,[1,1]}$ and $T_{1,[1,1]}$ respectively; let $P_0$ and $P_1$ be the $\delta$-functions of $T_{0,[0,1]}$ and $T_{1,[0,1]}$, respectively. The determinantal identities in this case read as follows 
\begin{align}
P_0 &= U_2U_0 - U_1^2, \\
P_1 &= U_3U_1 - U_2^2.
\end{align}
These relations may be regarded as first exchange relations of a cluster algebra, called $\mathcal{A}(\mathcal{C}_M)$ in \cite{GLS2}, with initial cluster $(U_0,U_1,P_0,P_1)$, initial exchange matrix visualized by the the quiver in Figure \ref{fig:in}
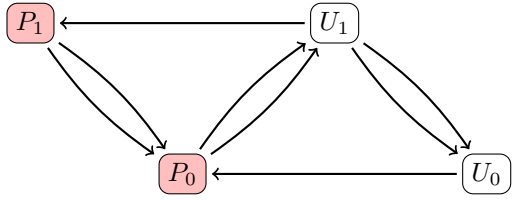
\begin{figure}[h!]
\begin{center}
\begin{tikzpicture}
\node[draw,rectangle,rounded corners,fill=red!25] at (0,2) (a) {$P_1$};
\node[draw,rectangle,rounded corners,fill=red!25] at (2,0) (b) {$P_0$};
\node[draw,rectangle,rounded corners] at (4,2) (c) {$U_1$};
\node[draw,rectangle,rounded corners] at (6,0) (d) {$U_0$};
\path[->,thick,shorten <=2pt,shorten >=2pt] (a) edge [bend left=10] node {} (b);
\path[->,thick,shorten <=2pt,shorten >=2pt] (a) edge [bend left=-10] node {} (b);
\path[->,thick,shorten <=2pt,shorten >=2pt] (b) edge [bend left=10] node {} (c);
\path[->,thick,shorten <=2pt,shorten >=2pt] (b) edge [bend left=-10] node {} (c);
\path[->,thick,shorten <=2pt,shorten >=2pt] (c) edge [bend left=10] node {} (d);
\path[->,thick,shorten <=2pt,shorten >=2pt] (c) edge [bend left=-10] node {} (d);
\path[->,thick,shorten <=2pt,shorten >=2pt] (c) edge node {} (a);
\path[->,thick,shorten <=2pt,shorten >=2pt] (d) edge node {} (b);
\end{tikzpicture}
\end{center}
\caption{Initial cluster} \label{fig:in}
\end{figure}
and frozen variables $P_0$ and $P_1$. The frozen variables $P_0$ and $P_1$ may be regarded as coefficients of the cluster algebra, see \cite{FZ3}.  

The mutations at $U_0$ and $U_1$ yield to the new cluster variables $U_2$ and $U_3$,
\begin{align*} U_2=\frac{U_1^2+P_0}{U_0}, && U_3=\frac{U_2^2+P_1}{U_1}.
\end{align*}
The associated quiver encodes the exchange relation; it gets also mutated. The mutated quivers are shown below.

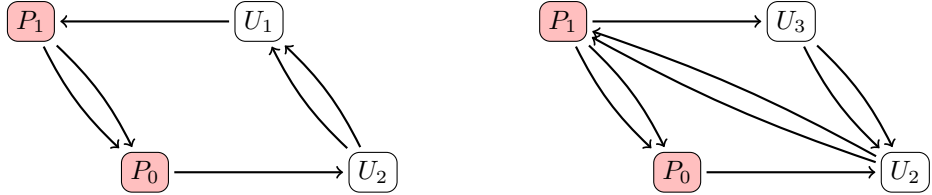
\begin{figure}
\begin{center}
\begin{tikzpicture}
\node[draw,rectangle,rounded corners,fill=red!25] at (0,2) (a) {$P_1$};
\node[draw,rectangle,rounded corners,fill=red!25] at (1.5,0) (b) {$P_0$};
\node[draw,rectangle,rounded corners] at (3,2) (c) {$U_1$};
\node[draw,rectangle,rounded corners] at (4.5,0) (d) {$U_2$};
\path[->,thick,shorten <=2pt,shorten >=2pt] (a) edge [bend left=10] node {} (b);
\path[->,thick,shorten <=2pt,shorten >=2pt] (a) edge [bend left=-10] node {} (b);
\path[->,thick,shorten <=2pt,shorten >=2pt] (d) edge [bend left=10] node {} (c);
\path[->,thick,shorten <=2pt,shorten >=2pt] (d) edge [bend left=-10] node {} (c);
\path[->,thick,shorten <=2pt,shorten >=2pt] (c) edge node {} (a);
\path[->,thick,shorten <=2pt,shorten >=2pt] (b) edge node {} (d);
\node[draw,rectangle,rounded corners,fill=red!25] at (7,2) (aa) {$P_1$};
\node[draw,rectangle,rounded corners,fill=red!25] at (8.5,0) (bb) {$P_0$};
\node[draw,rectangle,rounded corners] at (10,2) (cc) {$U_3$};
\node[draw,rectangle,rounded corners] at (11.5,0) (dd) {$U_2$};
\path[->,thick,shorten <=2pt,shorten >=2pt] (aa) edge [bend left=10] node {} (bb);
\path[->,thick,shorten <=2pt,shorten >=2pt] (aa) edge [bend left=-10] node {} (bb);
\path[->,thick,shorten <=2pt,shorten >=2pt] (dd) edge [bend left=5] node {} (aa);
\path[->,thick,shorten <=2pt,shorten >=2pt] (dd) edge [bend left=-5] node {} (aa);
\path[->,thick,shorten <=2pt,shorten >=2pt] (cc) edge [bend left=10] node {} (dd);
\path[->,thick,shorten <=2pt,shorten >=2pt] (cc) edge [bend left=-10] node {} (dd);
\path[->,thick,shorten <=2pt,shorten >=2pt] (aa) edge node {} (cc);
\path[->,thick,shorten <=2pt,shorten >=2pt] (bb) edge node {} (dd);
\end{tikzpicture}
\end{center}
\caption{The cluster after mutation at $U_0$ and $U_1$, consecutively}
\end{figure}

If we specialize the coefficients $P_0=P_1=1$ we get a coefficient-free cluster algebra of rank two with initial exchange matrix $B=\begin{pmatrix}
 0 & -2\\
2 & 0
\end{pmatrix}$; the specialized cluster variables $U_n (n \in \mathbb{Z})$ satisfy the recursion $$U_{n+1}U_{n-1}=U_n^2+1$$ for every $n \in \mathbb{Z}$. In \cite[Theorem 4.1]{CZ} Caldero and Zelevinsky proved that
\begin{align}
\label{bla}
U_{n+2} = \frac{1}{U_1^{n}U_0^{n+1}}\sum_{k,l} \genfrac(){0cm}{0}{n-k}{l}\genfrac(){0cm}{0}{n+1-l}{k}U_1^{2k}U_0^{2l} 
\end{align}
where the sum is taken over all $k,l \in \mathbb{N}$ such that either $k+l \leq n$ or $(k,l)=(n+1,0)$. Musiker and Propp (\cite{MP}) gave nice combinatorial descriptions of the coeffcients. The author (\cite{La}) gives a different formula for the coefficients. Sz\'ant\'o (\cite{Sz}) shows that a quantized version of the formula is related to the number of points in a Grassmannian over a finite field $\mathbb{F}_{q}$ in the context of {\it Hall algebras}.

Equation (\ref{bla}) illustrates the Fomin-Zelevinsky's {\it Laurent phenomenon} (\cite{FZ}): Every cluster variable $U_n (n \in \mathbb{Z})$ is a Laurent polynomial in $U_1$ and $U_0$.

Caldero and Zelevinsky (\cite{CZ}) derive formula (\ref{bla}) using the {\it Caldero-Chapoton map} from \cite{CC} and computing Euler characteristics of Grassmannians of quiver representations. Later Zelevinsky (\cite{Z}) gave a simpler proof for formula (\ref{bla}). He observed that the expression $$T=\frac{1+U_n^2+U_{n+1}^2}{U_nU_{n+1}}$$ is invariant of $n$. Thus, the non-linear exchange relation $U_{n+1}U_{n-1}=U_n^2+1$ may be replaced by a linear three-term recursion
\begin{align}
\label{threeterm}
U_{n+1}=TU_{n}-U_{n-1}, & \quad (n \in \mathbb{Z}), 
\end{align}
when we define $$T=\frac{1+U_1^2+U_2^2}{U_1U_2}.$$ Note that $T=U_3U_0-U_2U_1$.

In analogy to these formulae the cluster variables $U_n$ of the cluster algebra with initial exchange matrix $B=\begin{pmatrix}
 0 & 2\\
-2 & 0\\
0 & -1\\
1 & 0
\end{pmatrix}$, initial seed $(U_1,U_2,P_0,P_1)$, and frozen variables $P_0$ and $P_1$ satisfy the exchange relation
\begin{align}
\label{nonquantex}
U_{n+1}U_{n-1}=U_n^2+P_1^{n-1}P_0^{n-4}
\end{align}
for $n \geq 4$. The cluster variables are explicitly given by
\begin{align}\label{nonquant}
U_{n+3} = \frac{1}{U_1^{n+1}U_2^n}\sum_{k,l} \genfrac(){0cm}{0}{n-k}{l}\genfrac(){0cm}{0}{n+1-l}{k}P_1^{n+1-k}U_2^{2k}U_1^{2l}P_0^{n-l}.
\end{align} 
(For reasons that will become clear later on we have switched our initial seed from $(U_0,U_1)$ to $(U_1,U_2)$.)

Gei\ss -Leclerc-Schr\"{o}er realized the cluster algebra $\mathcal{A}(\mathcal{C}_M)$ as a subalgebra of the graded dual  $U(\mathfrak{n})_{gr}^{*}$ of the positive part $\mathfrak{n}$ of the universal enveloping algebra of the Kac-Moody Lie algebra of type $A_1^{(1)}$. 

A striking feature is {\it polynomiality}: Every $U_n$ is actually a polynomial in $U_3,U_2,U_1,U_0$. For example, one may check that $U_4=U_3^2U_0-2U_3U_2U_1+U_2^3$. A priori the cluster variables $U_n$ are only rational functions in $U_3,U_2,U_1,U_0$. If we plug $P_0=U_2U_0-U_1^2$ and $P_1=U_3U_1-U_2^2$ in equation (\ref{nonquant}) and use the binomial theorem we see that
\begin{align}\label{hugesum}
U_{n+3} = \sum_{a,b} c_{n,a,b} U_3^aU_2^{n+2-2a+b}U_1^{n-1-2b+a}U_0^b 
\end{align} 
with coefficients 
\begin{align}
\label{coeffs}
 c_{n,a,b}=\sum_{k,l}(-1)^{k+l+a+b+1} \genfrac(){0cm}{0}{n-k}{l}\genfrac(){0cm}{0}{n+1-l}{k}\genfrac(){0cm}{0}{n+1-k}{a}\genfrac(){0cm}{0}{n-l}{b}.
\end{align}
Polynomiality yields to the combinatorially non-trivial binomial identity $c_{n,a,b}=0$ if either $n+2-2a+b<0$ or $n-1-2b+a<0$.

\subsection{Mutation of rigid modules}
\label{rigid}
In the following subsection we describe the mutation of rigid $\Lambda$-modules. We use the abbreviations $T_0=T_{0,[0,0]}$, $T_1=T_{1,[0,0]}$, $P_0=T_{0,[0,1]}$, and $P_1=T_{1,[0,1]}$; recall that the $\Lambda$-modules on the right hand sides are displayed in Figure \ref{tmod}. 

The $\Lambda$-module $T=T_0 \oplus T_1 \oplus P_0 \oplus P_1$ is rigid. Moreover, $T$ is {\it maximal rigid}, i.e. every indecomposable $\Lambda$-module $\tilde{T}$ for which $T \oplus \tilde{T}$ is rigid is isomorphic to a direct summand of $T$. The quiver of $T$ is given in Figure \ref{fig:start}. Note the similarity with Figure \ref{fig:in}.

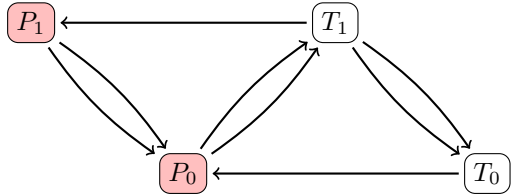
\begin{figure}[h!]
\begin{center}
\begin{tikzpicture}
\node[draw,rectangle,rounded corners,fill=red!25] at (0,2) (a) {$P_1$};
\node[draw,rectangle,rounded corners,fill=red!25] at (2,0) (b) {$P_0$};
\node[draw,rectangle,rounded corners] at (4,2) (c) {$T_1$};
\node[draw,rectangle,rounded corners] at (6,0) (d) {$T_0$};
\path[->,thick,shorten <=2pt,shorten >=2pt] (a) edge [bend left=10] node {} (b);
\path[->,thick,shorten <=2pt,shorten >=2pt] (a) edge [bend left=-10] node {} (b);
\path[->,thick,shorten <=2pt,shorten >=2pt] (b) edge [bend left=10] node {} (c);
\path[->,thick,shorten <=2pt,shorten >=2pt] (b) edge [bend left=-10] node {} (c);
\path[->,thick,shorten <=2pt,shorten >=2pt] (c) edge [bend left=10] node {} (d);
\path[->,thick,shorten <=2pt,shorten >=2pt] (c) edge [bend left=-10] node {} (d);
\path[->,thick,shorten <=2pt,shorten >=2pt] (c) edge node {} (a);
\path[->,thick,shorten <=2pt,shorten >=2pt] (d) edge node {} (b);
\end{tikzpicture}
\end{center}
\caption{The quiver of $T$} \label{fig:start}
\end{figure}

The dimension of the sixteen homomorphism spaces between the direct summands of $T$ are put in a matrix $C_T$ called {\it Cartan matrix}. In our example we have $C_T=\begin{pmatrix}1&2&3&4\\0&1&2&3\\1&2&4&6\\0&1&2&4\end{pmatrix}$, where rows and columns are ordered in accordance with the order $T_0$, $T_1$, $P_0$, $P_1$. For example ${\rm dim}\left({\rm Hom}(T_0,T_0)\right) =1$, ${\rm dim}\left({\rm Hom}(T_1,T_0)\right)=2$, ${\rm dim}\left({\rm Hom}(P_0,T_0)\right) =3$, etc.

There is a mutation process for maximal rigid $\Lambda$-modules analogous to the mutation process for cluster algebras. We refer to \cite{GLS1} for a detailed exposition. (\cite{GLS1} work with Dynkin quivers $Q$, but same procedures apply to the general setup as well.) The modules $P_0$ and $P_1$ are projective-injective and cannot be mutated; they correspond to frozen variables in cluster algebra. Both $T_0$ and $T_1$ can be mutated. Let us describe the mutation for $T_0$. There is a (unique up to isomorphism) $\Lambda$-module $T_2$ such that $T_2 \ncong T_0$ and the module $T_2 \oplus T/T_0 \cong T_2 \oplus T_1 \oplus P_0 \oplus P_1$ is again maximal rigid; two short exact sequences 
\begin{center}
\begin{tikzpicture}
\node at (0,0) (a) {$0$};
\node at (1.5,0) (b) {$T_0$};
\node at (3,0) (c) {$P_0$};
\node at (4.5,0) (d) {$T_2$};
\node at (6,0) (e) {$0$};
\path[->,shorten <=2pt,shorten >=2pt] (a) edge node {} (b);
\path[->,shorten <=2pt,shorten >=2pt] (b) edge node {} (c);
\path[->,shorten <=2pt,shorten >=2pt] (c) edge node {} (d);
\path[->,shorten <=2pt,shorten >=2pt] (d) edge node {} (e);
\node at (0,-0.5) (aa) {$0$};
\node at (1.5,-0.5) (bb) {$T_2$};
\node at (3,-0.5) (cc) {$T_1 \oplus T_1$};
\node at (4.5,-0.5) (dd) {$T_0$};
\node at (6,-0.5) (ee) {$0$};
\path[->,shorten <=2pt,shorten >=2pt] (aa) edge node {} (bb);
\path[->,shorten <=2pt,shorten >=2pt] (bb) edge node {} (cc);
\path[->,shorten <=2pt,shorten >=2pt] (cc) edge node {} (dd);
\path[->,shorten <=2pt,shorten >=2pt] (dd) edge node {} (ee);
\end{tikzpicture}
\end{center}
characterize $T_2$. The appearence of $P_0 $ and $T_1 \oplus T_1$ as middle terms is an incarnation of the fact that there is one arrow form $T_0$ to $P_0$ in the quiver of $T$ and two arrows form $T_1$ to $T_0$.  We see that $T_2=T_{0,[1,1]}$. 

Denote the mutated module by $T'=\mu_{T_0}(T)=T_2 \oplus T_1 \oplus P_0 \oplus P_1$. The Cartan matrix and the quiver of $T'$ are shown in Figure \ref{fig:onemut}. The quiver of $T'$ is obtained from the previous quiver by quiver mutation. A combinatorial recursion for the Cartan matrices is given in \cite[Proposition 7.5]{GLS1}.
\begin{figure}[h!]
\begin{center}
\begin{tikzpicture}
\node at (-4,1) (a) {$C_{T'}=\begin{pmatrix}1&0&1&2\\2&1&2&3\\3&2&4&6\\2&1&2&4\end{pmatrix}$};
\node[draw,rectangle,rounded corners,fill=red!25] at (0,2) (a) {$P_1$};
\node[draw,rectangle,rounded corners,fill=red!25] at (2,0) (b) {$P_0$};
\node[draw,rectangle,rounded corners] at (3,2) (c) {$T_1$};
\node[draw,rectangle,rounded corners] at (5,0) (d) {$T_2$};
\path[->,thick,shorten <=2pt,shorten >=2pt] (a) edge [bend left=10] node {} (b);
\path[->,thick,shorten <=2pt,shorten >=2pt] (a) edge [bend left=-10] node {} (b);
\path[->,thick,shorten <=2pt,shorten >=2pt] (d) edge [bend left=10] node {} (c);
\path[->,thick,shorten <=2pt,shorten >=2pt] (d) edge [bend left=-10] node {} (c);
\path[->,thick,shorten <=2pt,shorten >=2pt] (c) edge node {} (a);
\path[->,thick,shorten <=2pt,shorten >=2pt] (b) edge node {} (d);
\end{tikzpicture}
\end{center}
\caption{The Cartan matrix and the quiver of $T'$} \label{fig:onemut}
\end{figure}

Mutation of rigid modules is involutive just as mutation of cluster algebras, i.e. $\mu_{T_2}(T')=T$. Since $P_0$ and $P_1$ are not mutable, the only non-trivial further mutation is $T''=\mu_{T_1}(T')$. Using short exact sequences we see that $T''=T_2 \oplus T_3 \oplus P_0 \oplus P_1$ with $T_3=T_{1,[1,1]}$. Iteration of the mutation process yields to a sequence of $\Lambda$-modules $(T_n)_{n \in \mathbb{N}}$. Similarly to cluster variables, one obtains a sequence $(T_n)_{n \in \mathbb{N}^{-}}$ of $\Lambda$-modules by starting with mutation of $T$ at $T_1$.
\begin{figure}[h!]
\begin{center}
\begin{tikzpicture}
\node at (-4,1) (a) {$C_{T'}=\begin{pmatrix}1&2&1&2\\0&1&0&1\\3&4&4&6\\2&3&2&4\end{pmatrix}$};
\node[draw,rectangle,rounded corners,fill=red!25] at (0,2) (a) {$P_1$};
\node[draw,rectangle,rounded corners,fill=red!25] at (2,0) (b) {$P_0$};
\node[draw,rectangle,rounded corners] at (3,2) (c) {$T_3$};
\node[draw,rectangle,rounded corners] at (5,0) (d) {$T_2$};
\path[->,thick,shorten <=2pt,shorten >=2pt] (a) edge [bend left=0] node {} (b);
\path[->,thick,shorten <=2pt,shorten >=2pt] (a) edge [bend left=-10] node {} (b);
\path[->,thick,shorten <=2pt,shorten >=2pt] (c) edge [bend left=10] node {} (d);
\path[->,thick,shorten <=2pt,shorten >=2pt] (c) edge [bend left=-0] node {} (d);
\path[->,thick,shorten <=3pt,shorten >=3pt] (d) edge [bend left=5] node {} (a);
\path[->,thick,shorten <=3pt,shorten >=3pt] (d) edge [bend left=-5] node {} (a);
\path[->,thick,shorten <=2pt,shorten >=2pt] (a) edge node {} (c);
\path[->,thick,shorten <=2pt,shorten >=2pt] (b) edge node {} (d);
\end{tikzpicture}
\end{center}
\caption{The Cartan matrix and the quiver of $T'$} \label{fig:twomut}
\end{figure}

Using \cite[Proposition 7.5]{GLS1} and one proves by mathematical induction that for $n \geq 3$ the Cartan matrix of $T_n \oplus T_{n+1} \oplus P_0 \oplus P_1$ is $$\begin{pmatrix}(2n-5)^2 & (2n-4)^2-2 & 3n-9 & 5n-14 \\ (2n-4)^2 & (2n-3)^3 & 3n-6 & 5n-9 \\ 5n-11 & 5n-6 & 4 & 6 \\ 3n-6 & 3n-3 & 2 &4 \end{pmatrix}.$$ Especially, we have ${\rm dim}\left( {\rm End}(T_n) \right) = (2n-5)^2$.

\subsection{Bases of the cluster algebra $\mathcal{A}(\mathcal{C}_M)$}
\label{bases}
Several authors studied various bases of the cluster algebra $\mathcal{A}(\mathcal{C}_M)$. In this subsection we describe the {\it semicanonical basis} of Caldero-Zelevinsky (\cite{CZ}),  the {\it canonical basis} of Sherman-Zelevinsky (\cite{SZ}), and the {\it dual semicanonical basis} of Gei\ss -Leclerc-Schröer (\cite{GLS2}). To define these bases we introduce the {\it normalized Chebyshev polynomials of the first and second kind}.

\begin{defn}
Define a sequence $\left(T_k\right)_{k=0}^{\infty}$ of polynomials $T_k \in \mathbb{Z}[X]$ recursively by $T_0=2$, $T_1=X$, and $T_{k+1}=XT_k-T_{k-1}$ for $k \geq 1$; another sequence $\left(S_k\right)_{k=0}^{\infty}$ of polynomials $S_k \in \mathbb{Z}[X]$ is defined recursively by $S_0=1$, $S_1=X$, and $S_{k+1}=XS_k-S_{k-1}$ for $k \geq 1$. 
\end{defn}

The polynomial $T_k$ is called the $k^{th}$ {\it normalized Chebyshev polynomial of the first kind}; the polynomial $S_k$ is called the $k^{th}$ {\it normalized Chebyshev polynomial of the second kind}. Figure \ref{cheb} displays the Chebyshev polynomials with lowest indices.

\begin{figure}
\begin{center}
\begin{tabular}{|c||c|c|c|c|c|}\hline
$k$&$0$&$1$&$2$&$3$&$4$ \\ \hline\hline 
$T_k$&$2$&$x$&$x^2-2$&$x^3-3x$&$x^4-4x^2+2$ \\ \hline
$S_k$&$1$&$x$&$x^2-1$&$x^3-2x$&$x^4-3x^2+1$ \\ \hline
\end{tabular}
\end{center}
\caption{Normalized Chebyshev polynomials of the first and second kind}\label{cheb}
\end{figure}

A monomial $U_{n+1}^{a_1}U_n^{a_2}P_1^{a_3}P_0^{a_4}$ (with $a_1,a_2,a_3,a_4 \in \mathbb{N}$) in the cluster variables of a {\it single} cluster is called a {\it cluster monomial.} Let $\underline{Mono}$ be the set of all cluster monomials. Put $z=U_3U_0-U_2U_1=\frac{P_1P_0+P_1U_1^2+P_0U_2^2}{U_1U_2}$ and
\begin{itemize}
 \item $\underline{\mathcal{B}}=\underline{Mono}\cup\left\lbrace (P_1P_0)^{\frac{k}{2}}T_k\left(z(P_1P_0)^{-\frac{1}{2}}\right)  \vert k \geq 1\right\rbrace$, 
 \item $\underline{\mathcal{S}}=\underline{Mono}\cup\left\lbrace (P_1P_0)^{\frac{k}{2}}S_k\left(z(P_1P_0)^{-\frac{1}{2}}\right)  \vert k \geq 1\right\rbrace$, 
 \item $\underline{\Sigma}=\underline{Mono}\cup\left\lbrace  z^k  \vert k \geq 1 \right\rbrace$.
\end{itemize}
The elements $s_k=(P_1P_0)^{\frac{k}{2}}S_k\left(z(P_1P_0)^{-\frac{1}{2}}\right) \in \underline{\mathcal{S}}$ obey the relation 
\begin{align}s_{k+1}=zs_k-P_1P_0s_{k-1} \label{nq1}
\end{align}
 for $k \geq 2$ which resembles the relation 
\begin{align}U_{k+1}=zU_k-P_1P_0U_{k-1} \label{nq2}
\end{align}
for $k \geq 4$ for cluster variables. 

\begin{theorem}[\cite{CZ,SZ,GLS2}] 
Each of $\underline{\mathcal{B}}$, $\underline{\mathcal{S}}$, and $\underline{\Sigma}$ is a $\mathbb{Q}$-basis of $\mathcal{A}(\mathcal{C}_M)$.
\end{theorem}
The basis $\underline{\mathcal{B}}$ is known as the {\it canonical basis} of $\mathcal{A}(\mathcal{C}_M)$, $\underline{\mathcal{S}}$ is the {\it semicanonical basis} of $\mathcal{A}(\mathcal{C}_M)$, and $\underline{\Sigma}$ is the {\it dual semicanonical basis} of $\mathcal{A}(\mathcal{C}_M)$.

\section{Quantization}
\label{quant}
\subsection{The quantized universal enveloping algebra $U_q(\mathfrak{g})$ of type $A_1^{(1)}$}
Let $C=(a_{ij})_{1 \leq i,j \leq 2}=\begin{pmatrix}
 2 & -2\\
-2 & 2
\end{pmatrix}$ be the {\it Cartan matrix} associated with the Kronecker quiver $Q$. Let $\mathfrak{g}$ be the {\it Kac-Moody Lie algebra} of type $C$; it admits a triangular
decomposition $\mathfrak{g}=\mathfrak{n}_{-} \oplus \mathfrak{h} \oplus \mathfrak{n}$. The Lie algebra $\mathfrak{n}$ is called the {\it positive part} of $\mathfrak{g}$. 

The Lie algebra $\mathfrak{g}$ is studied by its {\it root lattice}. There are two {\it simple roots}, $\alpha_1$ and $\alpha_2$. By $\Delta^+$ we denote the set of {\it positive roots}. There are two kinds of positive roots called {\it real} and {\it imaginary} roots, i.e. $\Delta^+=\Delta^+_{re}\cup\Delta^+_{im}$ with $\Delta^+_{re}=\left\lbrace (n+1)\alpha_1+n\alpha_2 \colon n \in \mathbb{N} \right\rbrace \cup \left\lbrace n\alpha_1+(n+1)n\alpha_2 \colon n \in \mathbb{N} \right\rbrace$ and $\Delta^+_{im}=\left\lbrace n\alpha_1+n\alpha_2 \colon n \in \mathbb{N}^+ \right\rbrace$. Note that the real positive roots correspond to dimension vectors which admit a unique irreducible $kQ$-module. Examples are displayed in Figure \ref{fig:AR}. They are also the {\it g-vectors} of the cluster algebra of type $A_1^{(1)}$ introduced above, see \cite{FZ3}.

Let $W$ be the {\it Weyl group} of type $\mathfrak{g}$; it is generated by two simple reflections $s_1,s_2 \in W$ which act on the simple roots by
\begin{align*} s_1(\alpha_1)&=-\alpha_1, && s_1(\alpha_2) = \alpha_2 +2 \alpha_1, \\
 s_2(\alpha_1)&=\alpha_1 +2 \alpha_2, && s_2(\alpha_2) = -\alpha_2. 
\end{align*}
 
In Section \ref{sec:int} we mentioned the {\it universal enveloping algebra} $U(\mathfrak{n})$ of $\mathfrak{n}$. It is the $\mathbb{C}$-algebra generated by $E_1,E_2$ subject to the relations 
\begin{align*}
E_1^3E_2-3E_1^2E_2E_1+3E_1E_2E_1^2-E_2E_1^3=0,\\ 
E_2^3E_1-3E_2^2E_1E_2+3E_2E_1E_2^2-E_1E_2^3=0.
\end{align*}

These relations are called {\it Serre relations}.

It is known that $U(\mathfrak{n})$ can be endowed with a comultiplication $\Delta \colon U(\mathfrak{n}) \to U(\mathfrak{n}) \otimes U(\mathfrak{n})$ defined by $\Delta(x) = 1\otimes x + x \otimes 1$ for all $x \in \mathfrak{n}$ (using the canonical embedding $\iota \colon \mathfrak{n} \to U(\mathfrak{n})$) and an antipode so that $U(\mathfrak{n})$ becomes a cocommutative Hopf algebra. It is graded by the root lattice. The graded dual of $U(\mathfrak{n})$, the Hopf algebra $U(\mathfrak{n})_{gr}^*$, is a commutative $\mathbb{C}$-algebra.

By introducing a deformation parameter $q$ one can construct a series of Hopf algebras $U_q(\mathfrak{n})$ that are not cocommutative but specialize to $U(\mathfrak{n})$ if we set $q=1$. To describe this construction we introduce quantized integers and quantized binomial coefficients. 

Remarkably, the Hopf algebra $U_q(\mathfrak{n}) \cong U_q(\mathfrak{n})_{gr}^{*}$ is a self-dual Hopf algebra whereas $U(\mathfrak{n}) \not\cong U(\mathfrak{n})^{*}$.

\begin{defn}
For an integer $k$, let $$\left[k \right] =\frac{q^k-q^{-k}}{q-q^{-1}}\in \mathbb{Q}\left( q\right) $$ denote the quantum integer. For two integers $n,k$, let $$\genfrac[]{0cm}{0}{n}{k}=\frac{\left[n \right] \left[n-1 \right] \cdots \left[n -k+1\right]}{\left[k \right] \left[k-1 \right] \cdots \left[1\right]} \in \mathbb{Q}\left( q\right) $$ denote the quantum binomial coefficient. Furthermore, for a natural number $k$, let $$[k]!=[k][k-1]\cdots[1]$$ denote the quantized factorial. 
\end{defn}
Both $[k]$ and $\genfrac[]{0cm}{0}{n}{k}$ are actually Laurent polynomials in $q$. Note that $[k]=k$, $\genfrac[]{0cm}{0}{n}{k}=\genfrac(){0cm}{0}{n}{k}$, and $[k]!=k!$ if we specialize $q=1.$ Examples of quantum integers include
\begin{align*}
 [0]&=0, && &[1]=& \ 1, \\ [2]&=q+q^{-1}, && &[3]=& \ q^2+1+q^{-2}.
\end{align*}

Note that some authors, for example \cite{KC}, use a different notation for quantum integers. Note also that quantum binomial coefficients, just as ordinary binomial coefficients, are defined for negative integers $n,k$ as well. For example, $\genfrac[]{0cm}{0}{-2}{1}=-q-q^{-1}$, but $\genfrac[]{0cm}{0}{n}{k}=0$ if $k<0$. 

Quantized integers are related with the normalized Chebyshev polynomials $S_k$, for $k \geq 0$, of the second kind from Subsection \ref{bases}. More precisely, there holds $[k]=S_{k-1}([2])=S_{k-1}(q+q^{-1})$ for $k \geq 1$.

\begin{defn}
The quantized enveloping algebra $U_q(\mathfrak{g})$ is a $\mathbb{Q}(q)$-algebra generated by
\begin{align}
 E_i, (i=1,2), &&  F_i, (i=1,2), &&  K_i,K_i^{-1}, (i=1,2), \nonumber
\end{align}
subject to the following relations
\begin{align}
&K_iK_j =K_jK_i, && (i \neq j) \\
&K_iK_i^{-1} = K_i^{-1}K_i = 1, && (i=1,2)  \\
&K_iE_jK_i^{-1} = q^{a_{ij}}E_j, && (1 \leq i,j \leq 2) \label{nice}\\
&K_iF_jK_i^{-1} = q^{-a_{ij}}F_j, && (1 \leq i,j \leq 2) \\
&E_iF_j -F_jE_i = \delta_{ij}\frac{K_i-K_i^{-1}}{q-q^{-1}},&& (1 \leq i,j \leq 2) \label{commu}\\
&E_1^3E_2-[3]E_1^2E_2E_1+[3]E_1E_2E_1^2-E_2E_1^3=0, && \label{qs1}\\
&E_2^3E_1-[3]E_2^2E_1E_2+[3]E_2E_1E_2^2-E_1E_2^3=0, && \label{qs2} \\
&F_1^3F_2-[3]F_1^2F_2F_1+[3]F_1F_2F_1^2-F_2F_1^3=0, &&\\
&F_2^3F_1-[3]F_2^2F_1F_2+[3]F_2F_1F_2^2-F_1F_2^3=0, &&
\end{align}
where $\delta_{ij}$ is the Kronecker delta function.
\end{defn}
The subalgebra generated by $E_1$ and $E_2$ is called the {\it quantized enveloping algebra} $U_q(\mathfrak{n})$. The only relations in $U_q(\mathfrak{n})$ remain (\ref{qs1}) and (\ref{qs2}). These are called {\it quantized Serre relations}. The algebra $U_q(\mathfrak{n})$ specializes to $U(\mathfrak{n})$ in the limit $q=1$.

\subsection{The Poincar\'{e}-Birkhoff-Witt basis}
To construct a basis of $U_q(\mathfrak{n})$ Lusztig (\cite[Chapter 37]{Lu1}) constructs several {\it T-automorphisms}. We will use the notation $E_i^{(k)}=E_i^k/[k]!$ for $i=1,2$ and the similar notation for $F_i$. For every $i=1,2$ define
\begin{itemize}
\item $T_i(E_i)=-K_i^{-1}F_i$, 
\item $T_i(F_i)=-E_iK_i$,
\item $T_i(E_j)=\sum_{r+s=2}(-1)^{r}q^{-r}E_i^{(r)}E_jE_i^{(s)}$ for $j \neq i$,
\item $T_i(F_j)=\sum_{r+s=2}(-1)^{r}q^{r}F_i^{(s)}F_jF_i^{(r)}$ for $j \neq i$,
\item $T_i(K_j)=K_jK_i^{-a_{ij}}$ for $i=1,2$.
\end{itemize}
Lusztig shows that $T_i$ can be extended to an algebra homomorphism $$T_i \colon U_q(\mathfrak{g}) \to U_q(\mathfrak{g}).$$ (It is denoted $T'_{i,-1}$ in Lusztig's book \cite[Chapter 37]{Lu1} where the variable $q$ is called $v$ instead.) Furthermore, $T_i$ is an algebra automorphism. The images of the generators under the inverse $T_i^{-1}$ are given by (see \cite[Chapter 37]{Lu1}) 
\begin{itemize}
\item $T_i^{-1}(E_i)=-F_iK_i$, 
\item $T_i^{-1}(F_i)=-K_i^{-1}E_i$,
\item $T_i^{-1}(E_j)=\sum_{r+s=2}(-1)^{r}q^{-r}E_i^{(s)}E_jE_i^{(r)}$ for $j \neq i$,
\item $T_i^{-1}(F_j)=\sum_{r+s=2}(-1)^{r}q^{r}F_i^{(r)}F_jF_i^{(s)}$ for $j \neq i$,
\item $T_i^{-1}(K_j)=K_jK_i^{-a_{ij}}$ for $i=1,2$.
\end{itemize}

The automorphisms $T_i$ are sometimes called {\it braid operators}. This terminology comes from the fact that operators $T_i$ can be defined for abitrary quivers $Q$ and that they satisfy braid group relations. In this particular case it means that $T_1T_2$ has infinite order because $s_1s_2 \in W$ has infinite order. 

For every reduced expression $w=s_{i_1}s_{i_2}\cdots s_{i_k}$ of a Weyl group element $w \in W$ Lusztig (see \cite[Proposition 40.2.1]{Lu1}) constructs a {\it Poincar\'{e}-Birkhoff-Witt basis}. 

\begin{theorem}[Lusztig] 
Let $w \in W$ and let $s_{i_1}s_{i_2}\cdots s_{i_k}$ be a reduced expression for $w$. Then all elements $$p_i(\textbf{c}):= (T_{i_1} \circ T_{i_2} \circ \cdots \circ T_{i_{k-1}})(E_{i_k}^{(c_k)}) \cdots (T_{i_1} \circ T_{i_2})(E_{i_3}^{(c_3)})T_{i_1}(E_{i_2}^{(c_{2})})E_{i_1}^{(c_1)}$$ parametrized by sequences $\textbf{c}=(c_1,\cdots,c_k)\in \mathbb{N}^k$, form a $\mathbb{Q}(q)$-basis of a subalgebra called $U_q^+(w)$ of $U_q(\mathfrak{n})$ which does only depend on $w$ but not on the choice of the reduced expression for $w$. 
\end{theorem}
Let us make some remarks.
\begin{enumerate}
\item The basis $\left\lbrace p_i(\textbf{c}) \colon \textbf{c} \in \mathbb{N}^k \right\rbrace$ is called a {\it PBW-type basis} of $U_q^+(w)$.
\item The basis is not canonical in the sense that it depends on the choice of the reduced expression for $w$. Every choice of a reduced expression gives a bijection between $\mathbb{N}^k$ and a basis of $U_q^+(w)$. The bijections are known as {\it Lusztig parametrizations}.
\item The same theorem holds for other quivers. For a Dynkin quiver $Q$ there is a unique longest element $w_0 \in W$. In this case $U_q^+(w_0)=U_q(\mathfrak{n})$. Thus, in the Dynkin case we get a PBW basis of the whole algebra $U_q(\mathfrak{n})$ whereas in the present case we have to restrict ourselves to an appropriate subalgebra.
\item It is not obvious that $(T_{i_1} \circ T_{i_2} \circ \cdots \circ T_{i_{l-1}})(E_{i_l}^{(c_l)}) \in U_q(\mathfrak{n})$ for all $1\leq l \leq k$ since the T-automorphisms are maps $T_i \colon U_q(\mathfrak{g}) \to U_q(\mathfrak{g})$.
\item The algebra $U_q(\mathfrak{n})$ is graded by the root lattice $R$ if we set ${\rm deg}(E_i)=\alpha_i$ for $i=1,2$. Then ${\rm deg}((T_{i_1} \circ T_{i_2} \circ \cdots \circ T_{i_{l-1}})(E_{i_l}^{(c_l)}))=s_{i_1}s_{i_2}\cdots s_{i_{l-1}}(\alpha_{i_{l}})$ for all $1 \leq l \leq k$.
\end{enumerate}
Let us consider the reduced expression $s_1s_2s_1s_2$ associated with the terminal $kQ$-module $M$ defined in Section \ref{sec:int}. It is an $Q^{op}$-adapted reduced expression for the given orientation. Note that under the bijection between positive roots and dimension vectors of indecomposable modules given by {\it Kac's theorem}, the positive roots
\begin{itemize}
\item $\alpha_1$,
\item $s_1(\alpha_2)=2\alpha_1+\alpha_2$,
\item $s_1s_2(\alpha_1)=3\alpha_1+2\alpha_2$,
\item $s_1s_2s_1(\alpha_2)=4\alpha_1+3\alpha_2$
\end{itemize}
correspond to the dimension vectors of the four preinjective modules $I_0$, $I_1$, $\tau(I_0)$, $\tau(I_1)$ that are the direct summands of the terminal $kQ$-module $M$ from Section \ref{sec:int}. Therefore we introduce the notation
\begin{itemize}
\item $v_0=E_1$,
\item $v_1=T_1(E_2)$,
\item $v_2=(T_1 \circ T_2)(E_1)$,
\item $v_3=(T_1 \circ T_2 \circ T_1)(E_2)$.
\end{itemize}
Monomials of the form $v[\textbf{a}]=v_3^{(a_3)}v_2^{(a_2)}v_1^{(a_1)}v_0^{(a_0)}$ with $\textbf{a}=(a_3,a_2,a_1,a_0) \in \mathbb{N}^4$ form a basis of the subalgebra $U_q^+(s_1s_2s_1s_2)$. Denote this basis by $\mathcal{P}=\left\lbrace v[\textbf{a}] \colon \textbf{a} \in \mathbb{N}^4\right\rbrace$. We call the basis $\mathcal{P}$ the {\it PBW basis} of $U_q^+(s_1s_2s_1s_2)$.

\subsection{The derivation of the straightening relations using Lusztig's T-automorphisms}
The aim of this subsection is to write arbitrary monomials in the elements $v_3,v_2,v_1,v_0$, for example $v_0^7v_2^3$, as a $\mathbb{Q}(q)$-linear combination of basis elements $v[\textbf{a}]$. Clearly, $v_0^7v_2^3 \in U_q^+(s_1s_2s_1s_2)$ but $v_0^7v_2^3$ is not in $\mathcal{P}$ since $v_0$ and $v_2$ are multiplied in a different order.

First of all let us compute $v_1$. We have
\begin{align*}
v_1 &= T_1(E_2)\\ &=E_2E_1^{(2)}-q^{-1}E_1E_2E_1+q^{-2}E_1^{(2)}E_2  \\
&= \frac{1}{[2]}\Bigg(E_2E_1^2-(q^{-2}+1)E_1E_2E_1+q^{-2}E_1^2E_2 \Bigg). 
\end{align*}
It follows that
\begin{align*}
v_1v_0 - q^{2}v_0v_1 = \frac{1}{[2]} \Bigg( E_2E_1^3-(q^{-2}+1) E_1E_2E_1^2 + q^{-2} E_1^2 E_2 E_1  \\
- q^2 E_1E_2E_1^2 + (1+q^2) E_1^2E_2E_1-E_1^3E_2 \Bigg)  =0
\end{align*}
by the quantum Serre relation. Thus we know that $v_0v_1=q^{-2}v_1v_0$. Applying the composition $T_1 \circ T_2$ to this equation we get $v_2v_3=q^{-2}v_3v_2$. By an analogous argument with interchanged role of $E_1$ and $E_2$ we get $v_1v_2=q^{-2}v_2v_1$.

Note that we my write $v_1$ as a commutator $$v_1=\frac{1}{[2]}(E_2E_1-q^{-2}E_1E_2)E_1-E_1\frac{1}{[2]}(E_2E_1-q^{-2}E_1E_2)$$ of $E_1$ and an element which we will abbreviate as $$A=\frac{1}{[2]}(E_2E_1-q^{-2}E_1E_2).$$ For symmetry let us introduce an element $$B=\frac{1}{[2]}(E_1E_2-q^{-2}E_2E_1).$$ The next lemma is useful for further computations.

\begin{lemma}
The equation $T_1(A)=B$ holds. Similarly, $T_2(B)=A$.
\end{lemma}
\begin{proof}
We only prove the first statement. The second is similar. Note that
\begin{align}
T_1(E_1E_2-q^{-2}E_2E_1) &= -K_1^{-1}F_1 \Big(E_2E_1^2-(q^{-2}+1)E_1E_2E_1+q^{-2}E_1^2E_2 \Big) \nonumber \\
& \quad+q^{-2} \Big( E_2E_1^2-(q^{-2}+1)E_1E_2E_1+q^{-2}E_1^2E_2 \Big) K_1^{-1}F_1 \nonumber \\  
& =-K_1^{-1}F_1\Big(E_2E_1^2-(q^{-2}+1)E_1E_2E_1+q^{-2}E_1^2E_2 \Big)\nonumber \\
&\quad +K_1^{-1} \Big( E_2E_1^2-(q^{-2}+1)E_1E_2E_1+q^{-2}E_1^2E_2 \Big) F_1. \label{summe}
\end{align}
because $E_2K_1^{-1}=q^{-2}K_1^{-1}E_2$ and $E_1K_1^{-1}=q^2K_1^{-1}E_1$ by relation (\ref{nice}). Now we use (\ref{commu}) to deduce that
\begin{align}
E_2E_1^2F_1 &= F_1E_2E_1^2 + E_2 \frac{K_1-K_1^{-1}}{q-q^{-1}} E_1 + E_2E_1 \frac{K_1-K_1^{-1}}{q-q^{-1}}, \label{term1} \\
E_1E_2E_1F_1 &= F_1E_1E_2E_1 + \frac{K_1-K_1^{-1}}{q-q^{-1}} E_2 E_1 + E_1E_2 \frac{K_1-K_1^{-1}}{q-q^{-1}}, \label{term2} \\
E_1^2E_2F_1 &= F_1E_1^2E_2 + \frac{K_1-K_1^{-1}}{q-q^{-1}} E_1E_2 + E_1 \frac{K_1-K_1^{-1}}{q-q^{-1}}E_2, \label{term3}
\end{align}
The first terms on the RHS of equations (\ref{term1}),(\ref{term2}), and (\ref{term3}) cancel out with corresponding terms if we substitute in equation (\ref{summe}). We sum up the appropriate linear combinations of the remaining terms on the RHS of equations (\ref{term1}),(\ref{term2}) and (\ref{term3}). Using again relation (\ref{nice}) we get
\begin{align*}
T_1(E_1E_2-q^{-2}E_2E_1) &= \frac{K_1^{-1}K_1}{q-q^{-1}} \Bigg(q^2E_2E_1+E_2E_1-( q^{-2}+1) E_2E_1 \\ & \qquad -(q^{-2}+1)E_1E_2+q^{-2}E_1E_2+q^{-4}E_1E_2 \Bigg)\\
& \quad -\frac{K_1^{-1}K_1^{-1}}{q-q^{-1}} \Bigg( q^{-2}E_2E_1+E_2E_1-(q^{-2}+1)E_2E_1 \\ & \qquad-(q^{-2}+1)E_1E_2+q^{-2}E_1E_2+E_1E_2 \Bigg) \\
& =E_2E_1 -q^{-2}E_1E_2.
\end{align*}
\end{proof}

We know that $v_1=Av_0-v_0A$. Similarly, $T_2(E_1)=BE_2-E_2B$. Applying the automorphism $T_1$ yields to $v_2=Av_1-v_1A$. Applying $T_1 \circ T_2$ to the first equation yields to $v_3=Av_2-v_2A$. Thus, every $v_i$, for $1 \leq i \leq 3$, satifies the commutator relation $v_i=Av_{i-1}-v_{i-1}A$. 

\begin{lemma} The equations 
 \begin{align*}
 v_iv_{i+1} &= q^{-2}v_{i+1}v_i, &(0 \leq i \leq 2), \\
 v_iv_{i+2} &= q^{-2}v_{i+2}v_i+(q^{-2}-1)v_{i+1}^2, &(0 \leq i \leq 1), \\ 
 v_iv_{i+3} &= q^{-2}v_{i+3}v_i+(q^{-4}-1)v_{i+2}v_{i+1}, &(i=0). 
\end{align*} hold.
\end{lemma}

\begin{proof}
The equations in the first line have already been checked. Now
\begin{align*}
v_0v_2 &= v_0Av_1-v_0v_1A \\&=(Av_0-v_1)v_1-q^{-2}v_1(Av_0-v_1) \\ &= q^{-2}Av_1v_0-v_1^2 -q^{-2}v_1Av_0+q^{-2}v_1^2 \\&= q^{-2}v_2v_0 +(q^{-2}-1)v_1^2. 
\end{align*}
By interchanging the role of $E_1$ and $E_2$ and applying $T_1$ we also get $v_1v_3=q^{-2}v_3v_1+(q^{-2}-1)v_2^2$. These are the equations in the second line. Furthermore
\begin{align*}
v_0v_3 &= v_0Av_2-v_0v_2A \\ &= (Av_0-v_1)v_2-(q^{-2}v_2v_0+(q^{-2}-1)v_1^2)A \\&=Av_0v_2 -v_1v_2 -q^{-2}v_2v_0A+(q^{-2}-1)v_1^2A \\
&= A(q^{-2}v_2v_0+(q^{-2}-1)v_1^2)-q^{-2}v_2v_1-q^{-2}v_2v_0A-(q^{-2}-1)v_1^2A \\ &= q^{-2}(Av_2-v_2A)v_0+(q^{-2}-1)Av_1^2-(q^{-2}-1)v_1^2A.
\end{align*}
Now the rest of the lemma follows from
\begin{align*}
Av_1^2-v_1^2A &=(v_1A+v_2)v_1-v_1(Av_1-v_2) \\ &=v_2v_1+v_1v_2 = (q^{-2}+1)v_2v_1.
\end{align*}
\end{proof}

These relations are called {\it straightening relations}; they enable us to write every element in $U_q^+(s_1s_2s_1s_2)$ as a $\mathbb{Q}(q)$-linear combination of basis elements in $\mathcal{P}$, i.e. elements of the form $v_3^{(a_3)}v_2^{(a_2)}v_1^{(a_1)}v_0^{(a_0)}$ with $a_3,a_2,a_1,a_0 \in \mathbb{N}$ and coefficients in $\mathbb{Q}(q)$. 

The straightening relations tell us that $U_q^+(s_1s_2s_1s_2)$ becomes a {\it commutative} subalgebra of $U(\mathfrak{n})$ if we specialize $q=1$. This is remarkable because $U(\mathfrak{n})$ is a non-commutative algebra. (For instance, $E_1E_2 \neq E_2E_1$ in $U(\mathfrak{n})$.) The specialization $q=1$ is sometimes called the {\it classical limit}.

\section{The subalgebra $U_q^+(s_1s_2s_1s_2)$}
\label{sec:sub}
The definitions, results, and proofs from Subsection \ref{dualcan} and Lemma \ref{plemma} from Subsection \ref{sec:rec} are due to Leclerc (\cite{Le1}). Since \cite{Le1} is not published we give a brief sketch. 

\subsection{The dual canonical basis}
\label{dualcan}
To study the algebra $U_q^+(s_1s_2s_1s_2)$ Lusztig (\cite{Lu1}) and Kashiwara (\cite{Ka}) introduced (slightly different) non-degenerate bilinear forms $$\left(.,.\right) \colon U_q(\mathfrak{n}) \times U_q(\mathfrak{n}) \to \mathbb{Q}(q).$$ We work with Kashiwara's form. As described in \cite{Le2} the {\it dual PBW basis} is defined to be the basis adjoint to $\mathcal{P}$ with respect to the bilinear form. The generators $v_i$, for $0 \leq i \leq 3$, satisfy (compare \cite[Section 4.7]{Le2} and note the difference in sign conventions) $$\left(v_i,v_i\right)=\left(E_{(i+1)\alpha_1+i\alpha_2},E_{(i+1)\alpha_1+i\alpha_2}\right)=\frac{(1-q^{-2})^{2i+1}}{1-q^{-2}}=(1-q^{-2})^{2i}.$$ Therefore, the duals are given by $u_i=\frac{1}{(1-q^{-2})^{2i}}v_i$. We see that the $u_i$, $0 \leq i \leq 3$, satisfy the {\it same} straightening relations,
\begin{align*}
 u_iu_{i+1} &= q^{-2}u_{i+1}u_i, &(0 \leq i \leq 2), \\
 u_iu_{i+2} &= q^{-2}u_{i+2}u_i+(q^{-2}-1)u_{i+1}^2, &(0 \leq i \leq 1), \\ 
 u_iu_{i+3} &= q^{-2}u_{i+3}u_i+(q^{-4}-1)u_{i+2}u_{i+1}, &(i=0). 
\end{align*}
It is also possible to derive the straightening relations using Leclerc's algorithm from \cite{Le2} which features quantum shuffles.

Leclerc (\cite{Le1}) also introduced the following structures related to the algebra $U_q^+(s_1s_2s_1s_2)$. Define
\begin{itemize}
 \item a ring anti-automorphism $\sigma \colon U_q^+(s_1s_2s_1s_2) \to U_q^+(s_1s_2s_1s_2)$ by $\sigma(q)=q^{-1}$ and $\sigma(u_i)=q^{2i}u_i$ for $i \in \left\lbrace 0,1,2,3 \right\rbrace$, 
 \item a norm $N \colon \mathbb{N}^4 \to \mathbb{Z}$ by $N(a_3,a_2,a_1,a_0)=(a_3+a_2+a_1+a_0)^2-7a_3-5a_2-3a_1-a_0$,
 \item a partial order $\lhd$ on $\mathbb{N}^4$ by saying that $\textbf{a},\textbf{b} \in \mathbb{N}$ satisfy $\textbf{a} \lhd \textbf{b}$ if and only if $\textbf{b}-\textbf{a} \in \mathbb{N}\left(-1,2,-1,0 \right) \bigoplus \mathbb{N}\left(0,-1,2,-1 \right)$,
 \item a set $S(\textbf{a})=\left\lbrace \textbf{b} \in \mathbb{N}^4 \colon \textbf{a} \lhd \textbf{b} \ {\rm and} \ \textbf{a} \neq \textbf{b}\right\rbrace$,
 \item a function $b \colon \mathbb{N}^4 \to \mathbb{Z}$ by $b(a_3,a_2,a_1,a_0)=\genfrac(){0cm}{0}{a_3}{2}+\genfrac(){0cm}{0}{a_2}{2}+\genfrac(){0cm}{0}{a_1}{2}+\genfrac(){0cm}{0}{a_0}{2}$. 
\end{itemize}

Using these definitions one can describe the dual PBW basis and construct another basis of $U_q^+(w)$, the {\it dual canonical basis}. Both bases are parametrized by $\mathbb{N}^4$. For every $\textbf{a} =(a_3,a_2,a_1,a_0) \in \mathbb{N}^4$ the dual PBW basis element corresponding to $\textbf{a}$, $E[\textbf{a}]$, is given by $$E[\textbf{a}]=q^{b(\textbf{a})}u_3^{a_3}u_2^{a_2}u_1^{a_1}u_0^{a_0},$$ 
see \cite[Section 5.5]{Le2}. It is a rescaling of the PBW basis and for every $\textbf{a} \in \mathbb{N}^4$ we have $(E[\textbf{a}],v[\textbf{a}])=1$. In what follows we often use the fact that $N(\textbf{a})=N(\textbf{b})$ if $\textbf{b} \lhd \textbf{a}$. 
\begin{theorem}[Leclerc, \cite{Le1}]
\label{combprop}
There is a basis $\left\lbrace B[\textbf{a}] \colon \textbf{a} \in \mathbb{N}^4\right\rbrace$ of $U_q^+(w)$ such that for every $\textbf{a} \in \mathbb{N}^4$ the following two conditions hold
\begin{itemize}
 \item $B[\textbf{a}]-E[\textbf{a}] \in \bigoplus_{\textbf{b} \in S(\textbf{a})}q\mathbb{Z}[q]E[\textbf{b}]$,
 \item $\sigma(B[\textbf{a}])=q^{-N(\textbf{a})}B[\textbf{a}]$.
\end{itemize}
\end{theorem}

\begin{proof}
For every $k \in \mathbb{N}$ consider the set $W_k=\left\lbrace \textbf{a} \in \mathbb{N}^4 \colon a_3+a_2+a_1+a_0=k \right\rbrace$. Extend the partial order $\lhd$ on $W_k$ to a total order $<$ so that $\textbf{a}_1,\textbf{a}_2,\ldots,\textbf{a}_l$ are the elements of $W_k$.

We induct backwards. We can start with $B[\textbf{a}_l]=E[\textbf{a}_l]$. For the induction step, suppose that $\textbf{a}_{m+1},\textbf{a}_{m+2},\ldots,\textbf{a}_l$ satisfy the two conditions of Theorem \ref{combprop}. Expand $\sigma\left(E[\textbf{a}_m]\right)$ in the dual PBW basis using the straightening relations. We get a $\mathbb{Z}[q,q^{-1}]$-linear combination of basis elements $E[\textbf{b}]$ with $b \in \left\lbrace\textbf{a}_m,\textbf{a}_{m+1},\ldots,\textbf{a}_l\right\rbrace$, so $$\sigma\left(E[\textbf{a}_m]\right)=\sum_{i=m}^l c_i E[\textbf{a}_i]$$ with $c_{i}\in \mathbb{Z}[q,q^{-1}]$. A short calculation shows that $c_m=q^{-N(\textbf{a})}$; to get $E[\textbf{a}_m]$ you always have to choose the first summand when straightening a monomial.

By induction hypothesis we know that each $B[\textbf{a}_i]$, for $m+1\leq i\leq l$, is a $\mathbb{Z}[q,q^{-1}]$-linear combination in the elements $E[\textbf{a}_i],E[\textbf{a}_{i+1}],\ldots,E[\textbf{a}_l]$. The vector $\left(B[\textbf{a}_{m+1}],B[\textbf{a}_{m+2}],\ldots,B[\textbf{a}_l] \right)$ is obtained from $\left(E[\textbf{a}_{m+1}],\ldots,E[\textbf{a}_l] \right)$ by multiplication with an upper triangular matrix with diagonal entries $1$. By inverting we may write each $B[\textbf{a}_i]$, for $m+1\leq i\leq l$, as a $\mathbb{Z}[q,q^{-1}]$-linear combination of $B[\textbf{a}_i],B[\textbf{a}_{i+1}],\ldots,B[\textbf{a}_l]$. Thus, $$\sigma\left(E[\textbf{a}_m]\right)=q^{-N(\textbf{a}_m)}E[\textbf{a}_m]+\sum_{i=m+1}^ld_iB[\textbf{a}_i]$$ for some $d_{i}\in \mathbb{Z}[q,q^{-1}]$. Apply $\sigma$, an involution, to get $$E[\textbf{a}_m]=q^{N(\textbf{a}_m)}\sigma\left( E[\textbf{a}_m]\right) +\sum_{i=m+1}^l\sigma(d_i)q^{-N(\textbf{a}_i)}B[\textbf{a}_i].$$

The $B[\textbf{a}_i]$, for $m+1 \leq i \leq l$, are linearly independent. Multiply the first equation with $q^{N(\textbf{a}_i)}$ (and remember that $N(\textbf{a}_i)=N(\textbf{a}_m)$) to get $q^{N(\textbf{a}_i)}d_i=-q^{-N(\textbf{a}_i)}\sigma(d_i)=-\sigma(q^{N(\textbf{a}_i)}d_i)$. Therefore, there are polynomials $\phi_i \in q \mathbb{Z}[q]$ such that $q^{N(\textbf{a}_i)}d_i=\phi_i(q)-\phi_i(q^{-1})$. Now $B[\textbf{a}_m]=E[\textbf{a}_m]+\sum_{i=m+1}^l\phi_iB[\textbf{a}_i]$ satisfies the two conditions of Theorem \ref{combprop}. 
\end{proof}

The two conditions of Theorem \ref{combprop} imply that the basis $\left\lbrace B[\textbf{a}] \colon \textbf{a} \in \mathbb{N}^4\right\rbrace$ is adjoint to a basis $\left\lbrace b[\textbf{a}] \colon \textbf{a} \in \mathbb{N}^4\right\rbrace$ with respect to the bilinear form from above that satisfies the following two properties. On one hand we have $(b[\textbf{a}],b[\textbf{a}]) \in 1+q\mathbb{Z}[[q]]$ for every $\textbf{a} \in \mathbb{N}^4$. On the other hand we have $\overline{b[\textbf{a}]}=b[\textbf{a}]$ for every $\textbf{a} \in \mathbb{N}^4$. Here, the symbol $\bar{\quad}$ denotes the bar involution from \cite[Proposition 6]{Le2}. It follows from \cite[Theorem 14.2.3]{Lu1} that $\left\lbrace b[\textbf{a}] \colon \textbf{a} \in \mathbb{N}^4\right\rbrace$ is Lusztig's canonical basis. Therefore, $\left\lbrace B[\textbf{a}] \colon \textbf{a} \in \mathbb{N}^4\right\rbrace$ is the dual of the canonical basis, or the {\it dual canonical basis} to put it shortly.

The two conditions of Theorem \ref{combprop} uniquely determine the dual canonical basis. 

The simplest elements in the dual canonical basis are $B[1,0,0,0]=u_3$, $B[0,1,0,0]=u_2$, $B[0,0,1,0]=u_1$, and $B[0,0,0,1]=u_0$. Further examples include 
\begin{itemize}
\item $B[1,0,1,0]=u_3u_1-q^2u_2^2$,
\item $B[0,1,0,1]=u_2u_0-q^2u_1^2$,
\item $B[1,0,0,1]=u_3u_0-q^2u_2u_1$,
\item $B[2,0,0,1]=qu_3^2u_0-(q+q^3)u_3u_2u_1+q^5u_2^3$,
\item $B[1,0,0,2]=qu_3u_0^2-(q+q^3)u_2u_1u_0+q^5u_1^3$,
\item $B[2,0,0,2]=q^2u_3^2u_0^2-(2q^3+q^4)u_3u_2u_1u_0-q^6u_2^3u_0-q^6u_3u_1^3+q^8u_2^2u_1^2$.
\end{itemize}
Note that $B[1,0,1,0]$ and $B[0,1,0,1]$ are $q$-deformations of the elements $P_1$ and $P_0$. We introduce the abbreviations $p_1=u_3u_1-q^2u_2^2$ and $p_0=u_2u_0-q^2u_1^2$. As observed by Leclerc, the elements $p_0$ and $p_1$ have the remarkable property that they $q$-commute with each other and with each of the generators $u_3,u_2,u_1$ and $u_0$. More precisely, there holds $p_0p_1=q^{-4}p_1p_0$ and
\begin{align*}
&p_0u_0 =q^2u_0p_0, & &p_0u_1=u_1p_0, &  &p_0u_2=q^{-2}u_2p_0,  &p_0u_3=q^{-4}u_3p_0,  \\
&p_1u_0=q^4u_0p_1, & &p_1u_1=q^2u_1p_1, & &p_1u_2=u_2p_1,  &p_1u_3=q^{-2}u_3p_1. 
\end{align*}
These relations can be checked by elementary calculations using the straightening relations.

\subsection{A recursion for dual canonical basis elements}
\label{sec:rec}
Let us introduce the convention that $B[\textbf{a}]=0$ for some $\textbf{a} = (a_3,a_2,a_1,a_0) \in \mathbb{Z}^4$ if there is an $i \in \left\lbrace 0,1,2,3 \right\rbrace $ such that $a_i<0$. Note that $B[0,0,0,0]=1$. 

\begin{lemma}[Leclerc, \cite{Le1}]
\label{plemma}
For every $\textbf{a} =(a_3,a_2,a_1,a_0) \in \mathbb{N}^4$ we have
\begin{align*} B[a_3,a_2+1,a_1,a_0+1] &= q^{a_2+2a_1+3a_0}B[\textbf{a}]p_0 = q^{4a_3+3a_2+2a_1+a_0}p_0 B[\textbf{a}],\\
B[a_3+1,a_2,a_1+1,a_0] &= q^{3a_3+2a_2+a_1}p_1 B[\textbf{a}] = q^{a_3+2a_2+3a_1+4a_0}B[\textbf{a}]p_1.
\end{align*}
\end{lemma}
\begin{proof}
We only prove the equation $B[a_3+1,a_2,a_1+1,a_0] = q^{3a_3+2a_2+a_1}p_1 B[\textbf{a}]$. The other equations are similar. We prove that $q^{3a_3+2a_2+a_1}p_1 B[\textbf{a}]$ satisfies the two conditions of Theorem \ref{combprop}. Let us expand $B[\textbf{a}]$ in the dual PBW basis, i.e. $$B[\textbf{a}]=\sum c_{b_3,b_2,b_1,b_0}q^{b(b_3,b_2,b_1,b_0)}u_3^{b_3}u_2^{b_2}u_1^{b_1}u_0^{b_0}$$ where the sum is taken over all $\textbf{b}=(b_3,b_2,b_1,b_0) \in \mathbb{N}^4$ such that $\textbf{b} \lhd \textbf{a}$ and $c_{\textbf{b}} \in \mathbb{Z}[q]$. If $\textbf{b} \neq \textbf{a}$, then $c_{\textbf{b}} \in q\mathbb{Z}[q]$. 

Next, $u_3^{b_3}u_2^{b_2}u_1^{b_1}u_0^{b_0}p_1=q^{2b_3-2b_1-4b_0}p_1u_3^{b_3}u_2^{b_2}u_1^{b_1}u_0^{b_0}$. If $\textbf{b} \lhd \textbf{a}$, then $2b_3-2b_1-4b_0=2a_3-2a_1-4a_0$. Therefore,
\begin{align*}
\sigma\left(q^{3a_3+2a_2+a_1}p_1B[\textbf{a}]\right) &= q^{-3a_3-2a_2-a_1}q^{-N(1,0,1,0)}q^{-N(\textbf{a})}B[\textbf{a}]p_1 \\
&=q^{-3a_3-2a_2-a_1}q^{-N(1,0,1,0)}q^{-N(\textbf{a})}q^{2a_3-2a_1-4a_0}p_1B[\textbf{a}]\\
&=q^{-N(a_3+1,a_2,a_1+1,a_0)}q^{3a_3+2a_2+a_1}p_1B[\textbf{a}]. 
\end{align*}

Recall that $p_1=u_3u_1-q^2u_2^2$. Thus, $q^{3a_3+2a_2+a_1}p_1B[\textbf{a}]$ is equal to $$\sum c_{\textbf{b}}q^{b(\textbf{b})}q^{3a_3+2a_2+a_1}\left( u_3u_1u_3^{b_3}u_2^{b_2}u_1^{b_1}u_0^{b_0}-q^2u_2^2u_3^{b_3}u_2^{b_2}u_1^{b_1}u_0^{b_0}\right).$$ One can check by induction that for every positive integer $l$ there holds $$u_1u_3^l=q^{-2l}u_3^{l}u_1+\left(q^{-4l+2}-q^{-2l+2}\right)u_3^{l-1}u_2^2.$$
The sum above simplifies to
\begin{align*}
\sum &c_{\textbf{b}}q^{b(\textbf{b})}q^{3a_3+2a_2+a_1}\left(q^{-b_3}u_3^{b_3+1}u_1u_2^{b_2}u_1^{b_1}u_0^{b_0}-q^{-2b_3+2}u_3^{b_3}u_2^{b_2+2}u_1^{b_1}u_0^{b_0}\right)\\
&=\sum c_{\textbf{b}}q^{b(b_3+1,b_2,b_1+1,b_0)}u_3^{b_3+1}u_2^{b_2}u_1^{b_1+1}u_0^{b_0}\\ & - \sum c_{\textbf{b}}q^{b_3+b_1+1}q^{b(b_3,b_2+2,b_1,b_0)}u_3^{b_3}u_2^{b_2+2}u_1^{b_1}u_0^{b_0}.
\end{align*}
The coefficients $c_{\textbf{b}}$ are in $q\mathbb{Z}[q]$ except for $c_{\textbf{a}}=1$.
\end{proof}

The preceding lemma enables us two write every dual canonical basis element $B[a_3,a_2,a_1,a_0]$ as a product of powers of dual canonical basis elements $p_1$, $p_0$ and a power of the parameter $q$ and an element of the form $B[a_3,a_2,0,0]$, $B[0,a_2,a_1,0]$, $B[0,0,a_1,a_0]$ or $B[a_3,0,0,a_0]$. 

We have $B[a_3,a_2,0,0]=E[a_3,a_2,0,0]$, $B[0,a_2,a_1,0]=E[0,a_2,a_1,0]$ and $B[0,0,a_1,a_0]=E[0,0,a_1,a_0]$ because these dimension vectors are maximal elements with respect to the partial order $\lhd$. Therefore, dual canonical basis elements of the form $B[a_3,0,0,a_0]$ are particularly interesting. The elements $B[a_3,0,0,a_0]$ with $\vert a_3 -a_0 \vert \leq 1$ can be computed recursively.  
 
\begin{theorem}
\label{thm:recursion}
For every $n\geq 1$ the dual canonical basis elements parametrized by $(n,0,0,n-1),(n-1,0,0,n),(n,0,0,n)$ satisfy the following recursions 
\begin{align} B[n,0,0,n-1] &= q^{n-1} u_3  B\left[ n-1,0,0,n-1 \right]-q^{2n-1}u_2B\left[ n-1,0,1,n-2 \right]  \nonumber \\
&= q^{3n-3} B\left[ n-1,0,0,n-1 \right] u_3 -q^{2n-3}B\left[ n-1,0,1,n-2 \right] u_2, \nonumber \\[1em]
B[n-1,0,0,n] &= q^{n-1} B\left[ n-1,0,0,n-1 \right] u_0 -q^{2n-1}B\left[ n-2,1,0,n-1 \right] u_1  \nonumber \\
&= q^{3n-3} u_0 B\left[ n-1,0,0,n-1 \right]-q^{2n-3} u_1 B\left[ n-2,1,0,n-1 \right],\nonumber \\[1em]
B[n,0,0,n] &= q^{n-1} B\left[ n,0,0,n-1 \right] u_0 -q^{2n}B\left[ n-1,1,0,n-1 \right] u_1 \nonumber \\ 
&= q^{3n-1}u_0 B\left[ n,0,0,n-1 \right] -q^{2n-2} u_1 B\left[ n-1,1,0,n-1 \right] \nonumber \\
&= q^{n-1} u_3 B\left[ n-1,0,0,n \right] -q^{2n} u_2 B\left[ n-1,0,1,n-1 \right] \nonumber \\ 
&= q^{3n-1} B\left[ n-1,0,0,n \right] u_3 -q^{2n-2} B\left[ n-1,0,1,n \right] u_2.\nonumber
\end{align}
\end{theorem}

\begin{proof}
We prove the three statements simultaneously by mathematical induction on $n$. For $n=1$ the equations become $B[1,0,0,0]=u_3=u_3$, $B[0,0,0,1]=u_0=u_0$, and $B[1,0,0,1]=u_3u_0-q^2u_2u_1=q^2u_0u_3-u_1u_2$. Using the explicit formulae for $B[2,0,0,1],B[1,0,0,2]$ and $B[2,0,0,2]$ from above and the straightening relations one can check that the equations are true for $n=2$. Suppose that $n \geq 3$ and that the three statements are true for all smaller $n$. Define $$f=q^{n-1}u_3B[n-1,0,0,n-1]-q^{2n-1} u_2 B[n-1,0,1,n-2].$$ We claim that 
\begin{itemize}
 \item[(i)] $f-E[n,0,0,n-1] \in \bigoplus_{\textbf{b} \in S\left( (n,0,0,n-1) \right) } q \mathbb{Z} E[\textbf{b}]$,
 \item[(ii)] $\sigma (f)=q^{-N(n,0,0,n-1)} f$.
\end{itemize}
The two claims imply that $f=B[n,0,0,n-1]$. 

Let us expand $B[n-1,0,0,n-1]$ in the dual PBW basis, i.e. $$B[n-1,0,0,n-1]=\sum c_{a_3,a_2,a_1,a_0}q^{b(a_3,a_2,a_1,a_0)}u_3^{a_3}u_2^{a_2}u_1^{a_1}u_0^{a_0}$$ where the sum is taken over all $\textbf{a}=(a_3,a_2,a_1,a_0) \in \mathbb{N}^4$ such that $(n-1,0,0,n-1) \lhd \textbf{a}$. This implies that $a_3 \leq n-1$. By definition of the dual canonical basis all coefficients obey $c_{a_3,a_2,a_1,a_0} \in q \mathbb{Z}[q]$ except for $c_{n-1,0,0,n-1}=1$. Then
\begin{align*}
q^{n-1}u_3 B[n-1,0,0,n-1]= &\\ \sum c_{\textbf{a}}q^{-\genfrac(){0cm}{2}{a_3+1}{2}+\genfrac(){0cm}{2}{a_3}{2}+n-1}&q^{b(a_3+1,a_2,a_1,a_0)}u_3^{a_3+1}u_2^{a_2}u_1^{a_1}u_0^{a_0}.  
\end{align*}
But $q^{-\genfrac(){0cm}{2}{a_3+1}{2}+\genfrac(){0cm}{2}{a_3}{2}+n-1}=q^{n-1-a_3} \in \mathbb{Z}[q]$, so $c_{\textbf{a}}q^{n-1-a_3} \in q \mathbb{Z}[q]$ except for $q^{n-1-(n-1)}c_{n-1,0,0,n-1}=1$. 

Now let us expand $B[n-1,0,1,n-2]$ in the dual PBW basis, i.e. $$B[n-1,0,1,n-2]=\sum d_{a_3,a_2,a_1,a_0}q^{b(a_3,a_2,a_1,a_0)}u_3^{a_3}u_2^{a_2}u_1^{a_1}u_0^{a_0}$$ where the sum is taken over all $\textbf{a}=(a_3,a_2,a_1,a_0) \in \mathbb{N}^4$ such that $(n-1,0,1,n-2) \lhd \textbf{a}$. This implies that $a_3 \leq n-1$. By definition of the dual canonical basis all coefficients obey $d_{a_3,a_2,a_1,a_0} \in \mathbb{Z}[q]$.
\begin{align*}
q^{2n-1}u_2 B[n-1,0,1,n-2]&= \\ \sum d_{\textbf{a}}q^{-a_2+2n-1-2a_3}&q^{b(a_3,a_2+1,a_1,a_0)}u_3^{a_3}u_2^{a_2+1}u_1^{a_1}u_0^{a_0}.  
\end{align*}
Since $(n-1,0,1,n-2) \lhd \textbf{a}$, there exists non-negative integers $r,s$ such that $(a_3,a_2,a_1,a_0)=(n-1,0,1,n-2)+s(-1,2,-1,0)+r(0,-1,2,-1)$. Thus, $a_2=2s-r$ and $a_3=n-1-s$ so that $-a_2+2n-1-2a_3=r+1$. So $d_{\textbf{a}}q^{-a_2+2n-1-2a_3} \in q \mathbb{Z}[q]$.

Note that $N(n,0,0,n-1)=4n^2-12n+2 = \colon N$. We apply the anti-automorphism $\sigma$ to $f$ and use the fact that $B[n-1,0,0,n-1]$ and $B[n-1,0,1,n-2]$ are, up to a power of $q$, invariant under $\sigma$ to get
\begin{align*}
 q^N&\sigma(f) \\ 
&= q^{3n-3}B[n-1,0,0,n-1]u_3 - q^{2n-3} B[n-1,0,1,n-1] u_2\\
&= q^{3n-3}B[n-1,0,0,n-1]u_3 - q^{5n-9} p_1 B[n-2,0,0,n-2] u_2\\
&= q^{3n-3}\Bigg( q^{n-2}u_3 B[n-2,0,0,n-1]-q^{2n-2}u_2B[n-2,0,1,n-2] \Bigg) u_3\\
& \quad -q^{5n-9} p_1 \Bigg( q^{n-3}u_3 B[n-3,0,0,n-2] - q^{2n-4}u_2B[n-3,0,1,n-3] \Bigg) u_2 \\
&= q^{4n-5}u_3 B[n-2,0,0,n-1] u_3 -q^{5n-5}u_2 B[n-2,0,1,n-2] u_3\\
& \quad -q^{6n-12} p_1 u_3 B[n-3,0,0,n-2] - q^{7n-13} p_1 u_2 B[n-3,0,1,n-3] u_2. 
\end{align*}
The first and the third summand in the last expression add up to 
\begin{align*}
&q^{n-1}u_3 \Bigg( q^{3n-4}B[n-2,0,0,n-1]u_3-q^{5n-13}p_1 B[n-3,0,0,n-2]u_2\Bigg) \\
\quad &=q^{n-1}u_3 \Bigg( q^{3n-4}B[n-2,0,0,n-1]u_3-q^{2n-4}p_1 B[n-2,0,1,n-2]u_2\Bigg) \\
\quad &=q^{n-1}u_3B[n-1,0,0,n-1];
\end{align*}
whereas the second and the fourth summand add up to 
\begin{align*}
&q^{5n-5}u_2B[n-2,0,1,n-2]u_3-q^{7n-13}p_1u_2B[n-3,0,1,n-3]u_2 \\
&=q^{8n-14}u_2p_1B[n-3,0,0,n-2]u_3-q^{7n-13}p_1u_2B[n-3,0,1,n-3]u_2 \\
&=q^{5n-7} u_2p_1 \Bigg( q^{3n-7}B[n-3,0,0,n-1]u_3-q^{2n-6} B[n-3,0,1,n-3]u_2\Bigg) \\
&=q^{5n-7}u_2p_1B[n-2,0,0,n-2] \\
&=q^{2n-1}u_2B[n-1,0,1,n-2].
\end{align*}
Altogether we get $q^N\sigma(f)=q^{n-1}u_3B[n-1,0,0,n-1]-q^{2n-1}u_2B[n-1,0,1,n-2]=f$.

We see that $f=B[n,0,0,n-1]$ and that the equations of Theorem \ref{thm:recursion} hold. 

By the same argument one can show that
\begin{align*}
B[n-1,0,0,n] &= q^{n-1} B\left[ n-1,0,0,n-1 \right] u_0 -q^{2n-1}B\left[ n-2,1,0,n-1 \right] u_1  \nonumber \\
&= q^{3n-3} u_0 B\left[ n-1,0,0,n-1 \right]-q^{2n-3} u_1 B\left[ n-2,1,0,n-1 \right].\nonumber
\end{align*}

By a very similar argument with the established recursion for $B[n-1,0,0,n]$ and by the inductively known recursion for $B[n-2,0,0,n-1]$ one can show just as above that an element $$g=q^{n-1} u_3 B\left[ n-1,0,0,n \right] -q^{2n} u_2 B\left[ n-1,0,1,n-1 \right]$$
is indeed the dual canonical basis element $B[n,0,0,n]$ and derive the equations 
\begin{align*}
B[n,0,0,n] &= q^{n-1} u_3 B\left[ n-1,0,0,n \right] -q^{2n} u_2 B\left[ n-1,0,1,n-1 \right] \nonumber \\ 
&= q^{3n-1} B\left[ n-1,0,0,n \right] u_3 -q^{2n-2} B\left[ n-1,0,1,n \right] u_2;
\end{align*}
with the same technique one can derive the other equations
\begin{align*}
B[n,0,0,n] &= q^{n-1} u_3 B\left[ n-1,0,0,n \right] -q^{2n} u_2 B\left[ n-1,0,1,n-1 \right] \nonumber \\ 
&= q^{3n-1} B\left[ n-1,0,0,n \right] u_3 -q^{2n-2} B\left[ n-1,0,1,n \right] u_2.
\end{align*}
\end{proof}

From the previous lemma we get a corollary.

\begin{cor}
\label{cor}
 For every integer $n \geq 1$ the following equations hold
\begin{align*}
 &B[n,0,0,n-1]B[1,0,0,1] = q^{3-4n}B[n+1,0,0,n]+q^{4-4n}B[n,1,1,n-1], \\
 &B[1,0,0,1]B[n,0,0,n-1] = q^{1-4n}B[n+1,0,0,n]+q^{-4n}B[n,1,1,n-1],\\[0.5em]
 &B[n,0,0,n]B[1,0,0,1] = q^{-4n}B[n+1,0,0,n+1]+q^{-4n}B[n,1,1,n], \\
 &B[1,0,0,1]B[n,0,0,n] = q^{-4n}B[n+1,0,0,n+1]+q^{-4n}B[n,1,1,n].
\end{align*}
\end{cor}
The last two equations were conjectured by Leclerc. 

\begin{proof}
We prove the statement by mathematical induction. The case $n=1$ can be checked in a short calculation using the straightening relations. Note also that the last two equations in Corollary \ref{cor} make sense and are true for $n=0$. Suppose that the equations hold for $n$ and all smaller numbers. Let us write down the first equation of Theorem (\ref{thm:recursion}) for three consecutive integers $n+1$, $n$, and $n-1$,
\begin{align}
&B[n+1,0,0,n]  = q^{n}u_3B[n,0,0,n]-q^{2n+1}u_2B[n,0,1,n-1], \label{r1} \\
&B[n,0,0,n-1]  = q^{n-1}u_3B[n-1,0,0,n-1] -q^{2n-1}u_2B[n-1,0,1,n-2], \label{r2} \\
&B[n-1,0,0,n-2] = q^{n-2}u_3B[n-2,0,0,n-2]-q^{2n-3}u_2B[n-2,0,1,n-3]. \label{r3} 
\end{align}
Multiply (\ref{r3}) from the right by $q^{4n-5}p_0p_1$to get
\begin{align}
q^{4-4n}&B[n,1,1,n-1] \nonumber \\ &= q^{-3n+3}u_3B[n-1,1,1,n-1]-q^{n+1}u_2p_1B[n-2,1,1,n-2], 
\end{align}
multiply (\ref{r2}) from the right by $B[1,0,0,1]$ to get
\begin{align}
B[n,0,0,n-1]&B[1,0,0,1] \nonumber \\ &=q^{n-1} u_3B[n-1,0,0,n-1]B[1,0,0,1] \nonumber \\ & \quad -q^{5n-3}u_2p_1B[n-2,0,0,n-2]B[1,0,0,1],
\end{align}
multiply (\ref{r1}) by $q^{3-4n}$ to get
\begin{align}
 q^{3-4n}&B[n+1,0,0,n] \nonumber \\ &= q^{3-3n}u_3B[n,0,0,n] -q^{n+1}u_2p_1B[n-1,0,0,n-1].
\end{align}
Using the induction hypothesis for $n-1$ and $n$ we see that $B[n,0,0,n-1]B[1,0,0,1] = q^{3-4n}B[n+1,0,0,n]+q^{4-4n}B[n,1,1,n-1]$. The other equations in Corollary \ref{cor} can be proved in a similar way. 
\end{proof}

The last two equations in Corollary \ref{cor} imply that the dual canonical basis element $B[1,0,0,1]$ commutes with every $B[n,0,0,n]$ $(n \in \mathbb{Z})$. A conjecture by Berenstein and Zelevinsky (see \cite{BZ}) says that the product of two dual canonical basis elements $b_1$ and $b_2$ is, up to a power of $q$, again a dual canonical basis element if and only if $b_1$ and $b_2$ $q$-commute, that is $b_1b_2=q^{s}b_2b_1$ for some $s \in \mathbb{Z}$. The conjecture turns out to be wrong. Using his quantum shuffle algorithm from \cite{Le2}, in \cite{Le3} Leclerc constructs five counterexamples. The last two equations of Corollary \ref{cor} give infinitely many counterexamples to Berenstein and Zelevinsky's conjecture. The commutativity of $B[1,0,0,1]$ and $B[n,0,0,n]$ implies $q$-commutativity, but the product $B[n,0,0,n]B[1,0,0,1]$ is a linear combination of {\it two} dual canonical basis elements. 

Several authors, e.g. Reineke in \cite{Re}, emphasized the importance of multiplicative properties of dual canoncial basis elements. Corollary \ref{cor} gives four series of expansions of products of dual canonical basis elements in the dual canonical basis.

\section{The quantum cluster algebra structure}
\subsection{The quantized version of the explicit formula for cluster variables}

We may view Corollary \ref{cor} as a quantization of the formulae (\ref{nq1}) and (\ref{nq2}) from Subsection \ref{bases}. Define the integral form $\mathcal{A}_{\mathbb{Z}}$ of $U_q^{+}(w)$ as $$\mathcal{A}_{\mathbb{Z}}=\bigoplus_{{\textbf a} \in \mathbb{N}^4}\mathbb{Z}[q,q^{-1}]u[{\textbf a}];$$ define an algebra $\mathcal{A}_1$ to be $$\mathcal{A}_1=\mathbb{Q} \otimes_{\mathbb{Z}[q,q^{-1}]} \mathcal{A}_{\mathbb{Z}}.$$ Then $\mathcal{A}_1=\mathbb{Q}[U_0,U_1,U_2,U_3]$ with $U_i=1 \otimes u_i$ for $i=0,1,2,3$. We see that $\mathcal{A}_1=\mathcal{A}(\mathcal{C}_M)$.

Note that $B[1,0,0,1]=u_3u_0-q^2u_2u_1 \in \mathcal{A}_{\mathbb{Z}}$ becomes $1 \otimes B[1,0,0,1]=z \in \mathcal{A}_1$ in the specialization $q=1$. The elements $p_1=u_3u_1-q^2u_2^2$ and $p_0=u_2u_0-q^2u_1^2$ specialize to $P_1=U_3U_1-U_2^2$ and $P_0=U_2U_0-U_1^2$. Corollary \ref{cor} means that the elements $B[n,0,0,n-1]$ are quantized cluster variables, because $B[n,1,1,n-1]$ is equal to $B[n-1,0,0,n-2]p_1p_0$ up to a power of $q$. Similarly, the specialization of $B[n,0,0,n]$ at $q=1$ is an element in Caldero-Zelevinsky's semicanonical basis of $\mathcal{A}(\mathcal{C}_M)$, because Corollary \ref{cor} provides a quantized version of the Chebyshev recursion (\ref{nq1}).

In the rest of this section we want to study the quantized cluster algebra structure of $U_q^+(w)$. We will work with quantum binomial coefficients instead of ordinary binomial coefficients as in Section \ref{sec:int}. The next lemma is a quantized version of the addition rule in Pascal's triangle.

\begin{prop}
\label{quantpas}
For quantum binomial coefficients the following relation holds 
\begin{eqnarray}
\genfrac[]{0cm}{0}{n}{k} &=& q^k \genfrac[]{0cm}{0}{n-1}{k}+q^{k-n}\genfrac[]{0cm}{0}{n-1}{k-1} \nonumber \\
&=& q^{-k} \genfrac[]{0cm}{0}{n-1}{k}+q^{n-k}\genfrac[]{0cm}{0}{n-1}{k-1}. \nonumber
\end{eqnarray}
\end{prop}
\begin{proof}
 See \cite{KC}, p. 17-18. 
\end{proof}

\begin{defn}
Define two functions $f,g \colon \mathbb{Z}^3 \to \mathbb{Z}$ by 
\begin{eqnarray}
 f(n,k,l)=n(n-2)+k(n+2)+l(n+1)-2kl, \nonumber \\
 g(n,k,l)=n(n-3)+k(n+1)+l(n+1)-2kl. \nonumber
\end{eqnarray}
\end{defn}

\begin{theorem}
\label{thm:clu}
For every natural number $n\geq 0$ we have 
\begin{align}
& u_2^nB\left[n+1,0,0,n \right]u_1^{n+1} = \sum_{k,l}q^{f(n,k,l)} \genfrac[]{0cm}{0}{n-k}{l}\genfrac[]{0cm}{0}{n+1-l}{k}p_1^{n+1-k}u_2^{2k}u_1^{2l}p_0^{n-l}, \label{clu1} \\
& u_2^nB\left[n,0,0,n \right]u_1^{n} = \sum_{k,l}q^{g(n,k,l)} \genfrac[]{0cm}{0}{n-k}{l}\genfrac[]{0cm}{0}{n-l}{k}p_1^{n-k}u_2^{2k}u_1^{2l}p_0^{n-l}. \label{clu2}
\end{align}
The summation in the first case runs over all pairs $(k,l) \in \mathbb{N}^2$ such that $k+l \leq n$ or $(k,l)=(n+1,0)$; the summation in the second case runs over all pairs $(k,l) \in \mathbb{N}^2$ such that $k+l \leq n$. 
\end{theorem}

\begin{proof}
We prove the theorem by mathematical induction. One can check that both equations hold for $n=0$ and $n=1$. Let $n \geq 2$ and suppose that the equations hold for all smaller values of $n$. By Theorem \ref{thm:recursion} we have 
\begin{align*}
B[n,0,0,n]&=q^{n-1}B[n,0,0,n-1]u_0-q^{2n}B[n-1,1,0,n-1]u_1 \\ &= q^{n-1}B[n,0,0,n-1]u_0-q^{5n-6}B[n-1,0,0,n-2]p_0u_1. 
\end{align*}
In the following calculations we intensively use the fact that the four variables $p_1,p_0,u_2,u_1$ $q$-commute with each other, see Subsection \ref{dualcan}. By induction hypothesis we can assume that 
\begin{align}
\label{erstesumme}
 u_2^n&q^{n-1}B[n,0,0,n-1]u_0u_1^n \nonumber \\&=q^{-n-1}u_2^nB[n,0,0,n-1]u_1^nu_0 \nonumber \\
&=\sum_{k,l} q^{f(n-1,k,l)-n-1}\genfrac[]{0cm}{0}{n-1-k}{l}\genfrac[]{0cm}{0}{n-l}{k}u_2p_1^{n-k}u_2^{2k}u_1^{2l}p_0^{n-1-l}u_0 \nonumber \\
&=\sum_{k,l} q^{f(n-1,k,l)+n-3+2l}\genfrac[]{0cm}{0}{n-1-k}{l}\genfrac[]{0cm}{0}{n-l}{k}p_1^{n-k}u_2^{2k}u_1^{2l}p_0^{n-1-l}u_2u_0.
\end{align}
Now we use the identity $u_2u_0=p_0+q^2u_1^2$. The sum (\ref{erstesumme}) splits into two summmands, namely
\begin{align}
\label{s1}
\sum_{k,l} q^{f(n-1,k,l)+n-3+2l}\genfrac[]{0cm}{0}{n-1-k}{l}\genfrac[]{0cm}{0}{n-l}{k}p_1^{n-k}u_2^{2k}u_1^{2l}p_0^{n-l},
\end{align}
and
\begin{align}
\label{s2}
&\sum_{k,l} q^{f(n-1,k,l)+n-1+2l}\genfrac[]{0cm}{0}{n-1-k}{l}\genfrac[]{0cm}{0}{n-l}{k}p_1^{n-k}u_2^{2k}u_1^{2l+2}p_0^{n-1-l} \nonumber \\
&\quad =\sum_{k,l} q^{f(n-1,k,l)+n-3+2l}\genfrac[]{0cm}{0}{n-1-k}{l-1}\genfrac[]{0cm}{0}{n+1-l}{k}p_1^{n-k}u_2^{2k}u_1^{2l}p_0^{n-l}. 
\end{align}
In the last step we shifted the index from $l$ to $l-1$. Again by induction hypothesis we have
\begin{align}
\label{s3}
u_2^n&q^{5n-6}B[n-1,0,0,n-1]p_0u_1^n \nonumber \\ 
&= \sum_{k,l} q^{f(n-2,k,l)+5n-6}\genfrac[]{0cm}{0}{n-2-k}{l}\genfrac[]{0cm}{0}{n-1-l}{k}p_1^{n-1-k}u_2^{2k+2}u_1^{2l+2}p_0^{n-1-l} \nonumber \\
&= \sum_{k,l} q^{f(n-2,k-1,l-1)+5n-6}\genfrac[]{0cm}{0}{n-1-k}{l-1}\genfrac[]{0cm}{0}{n-l}{k-1}p_1^{n-k}u_2^{2k}u_1^{2l}p_0^{n-l}.
\end{align}
A calculation shows that $f(n-1,k,l)-g(n,k,l)=-n+3-l$, $f(n-1,k,l-1)-g(n,k,l)=-2n+3-l$ and $f(n-2,k-1,l-1)-g(n,k,l)=-5n+7+k$. Thus, by comparing coefficients in (\ref{s1}), (\ref{s2}) and (\ref{s3}) it is enough to show that 
\begin{align*} 
\genfrac[]{0cm}{0}{n-k}{l}\genfrac[]{0cm}{0}{n-l}{k} = q^l \genfrac[]{0cm}{0}{n-1-k}{l}\genfrac[]{0cm}{0}{n-l}{k} & +q^{-n+l}\genfrac[]{0cm}{0}{n-1-k}{l-1}\genfrac[]{0cm}{0}{n+1-l}{k-1} \\ & -q^{k+1} \genfrac[]{0cm}{0}{n-1-k}{l-1}\genfrac[]{0cm}{0}{n-l}{k-1}. 
\end{align*}
But, by Proposition \ref{quantpas},
\begin{align*}
& \genfrac[]{0cm}{0}{n-k}{l}\genfrac[]{0cm}{0}{n-l}{k} - q^l \genfrac[]{0cm}{0}{n-1-k}{l}\genfrac[]{0cm}{0}{n-l}{k} \\ 
& \quad =\genfrac[]{0cm}{0}{n-l}{k} \Bigg( \genfrac[]{0cm}{0}{n-k}{l} - q^l \genfrac[]{0cm}{0}{n-k}{l} \Bigg) = q^{l+k-n} \genfrac[]{0cm}{0}{n-l}{k} \genfrac[]{0cm}{0}{n-k-1}{l-1},
\end{align*}
and
\begin{align*}& q^{-n+l}\genfrac[]{0cm}{0}{n-1-k}{l-1}\genfrac[]{0cm}{0}{n+1-l}{k} - q^{k+1} \genfrac[]{0cm}{0}{n-1-k}{l-1}\genfrac[]{0cm}{0}{n-l}{k-1} \\
& \quad = q^{n-l}\genfrac[]{0cm}{0}{n-k-1}{l-1} \Bigg( \genfrac[]{0cm}{0}{n+1-l}{k} - q^{l+k-n+1} \genfrac[]{0cm}{0}{n-l}{k-1} \Bigg) \\
& \quad = q^{k+l-n} \genfrac[]{0cm}{0}{n-1-k}{l-1}\genfrac[]{0cm}{0}{n-l}{k}.
\end{align*}
This proves (\ref{clu1}). Once we have established equation (\ref{clu1}), equation (\ref{clu2}) can be proved similarly using a recursion for $B[n+1,0,0,n]$ from Theorem \ref{thm:recursion} involving $B[n,0,0,n]$ and $B[n,0,1,n-1]$. 
\end{proof}

Theorem \ref{thm:clu} is a quantized version of the formula (\ref{nonquant}) for the cluster variables in Section \ref{sec:int}. There is an analogous formula for $B[n,0,0,n+1]$. 

It is possible to give an explicit expansion of dual canonical elements of the form $B[n,0,0,n-1]$ in the (dual) PBW basis. Therefore we have to write powers of $p_1$ and $p_0$ in the (dual) PBW basis. The following relations can be proved by induction. For every natural number $k$ the relations
\begin{align}
 p_1^k =\sum_{i=0}^k (-1)^iq^{2i^2-ik-k^2+i+k}\genfrac[]{0cm}{0}{k}{i}u_3^{k-i}u_2^{2i}u_1^{k-i}, \label{pot1}\\
 p_0^k =\sum_{i=0}^k (-1)^iq^{2i^2-ik-k^2+i+k}\genfrac[]{0cm}{0}{k}{i}u_2^{k-i}u_1^{2i}u_0^{k-i} \label{pot2}
\end{align}
hold. Substituting (\ref{pot1}) and (\ref{pot2}) in (\ref{clu1}) yields to
\begin{align*}
 B[n+1,0,0,n]=\sum_{k,l,r,s} (-1)^{k+l+s+r+1} \genfrac[]{0cm}{0}{n-k}{l}\genfrac[]{0cm}{0}{n+1-l}{k}\genfrac[]{0cm}{0}{n+1-k}{s}\genfrac[]{0cm}{0}{n-l}{r} \\
\cdot q^{-l-2kl+2n+kn+ln-3r-lr-r^2+s-ks+2rs-s^2} \\ \cdot E[s,n+2-2s+r,n-1-2r+s,r],
\end{align*}
here the sum is taken over all $k,l,r,s \in \mathbb{N}$ such that $0 \leq s \leq n+1-k$, $0 \leq r \leq n-l$ and either $k+l \leq n$ or $(k,l)=(n+1,0)$. 

The formula is an $q$-analogue of (\ref{coeffs}).

\subsection{The quasi-commutativity of adjacent quantized cluster variables and the quantum exchange relation}

We prove that adjacent quantized cluster quasi-commute, i.e. they are commutative up to a power of $q$.

\begin{lemma}
\label{qcomm}
For every $n \geq 1$ the elements $B[n+1,0,0,n]$ and $B[n,0,0,n-1]$ are $q$-commutative. More precisely, $$B[n,0,0,n-1]B[n+1,0,0,n]=q^2B[n+1,0,0,n]B[n,0,0,n-1].$$
\end{lemma}

\begin{proof}
We prove the theorem by mathematical induction. We can verify the statement for $n=1$ in a short calculation using the straightening relations. Suppose that the statement holds for $n-1$. Combine Lemma \ref{plemma} with Corollary \ref{cor} to get
\begin{align*}
B[n+1,0,0,n] &= q^{4n-1}B[1,0,0,1]B[n,0,0,n-1] \\ & \quad -q^{8n-14}B[n-1,0,0,n-2]p_1p_0 \\
&= q^{4n-3} B[n,0,0,n-1]B[1,0,0,1] \\ & \quad -q^{8n-12}B[n-1,0,0,n-2]p_1p_0. 
\end{align*}
Multiply the first expression for $B[n+1,0,0,n]$ from the left and the second from the right with $B[n,0,0,n-1]$. It remains to show that $$B[n,0,0,n-1]B[n-1,0,0,n-2]p_1p_0=q^4B[n-1,0,0,n-2]p_1p_0B[n,0,0,n-1],$$ 
which follows from the induction hypothesis and the relation $$B[n,0,0,n-1]p_1p_0=q^6p_1p_0B[n,0,0,n-1],$$
which follows from Lemma \ref{plemma}.
\end{proof}

Lemma \ref{qcomm} says that every two adjacent quantized cluster variables $B[n+1,0,0,n]$ and $B[n,0,0,n-1]$ form a {\it quantum torus}. If we specialize $q=1$, the elements $B[n+1,0,0,n]$ and $B[n,0,0,n-1]$ become cluster variables in the same cluster of $\mathcal{A}(\mathcal{C}_M)$. 

\begin{lemma}
\label{qexch}
For $n \geq 2$ we have $$B[n+1,0,0,n]B[n-1,0,0,n-2]=q^2B[n,0,0,n-1]^2+q^{2n^2-6n+8}p_1^{n+1}p_0^{n-2}.$$ 
\end{lemma}

\begin{proof}The statement is true in the case $n=2$. We use mathematical induction. Consider the product $p=B[n,0,0,n-1]B[1,0,0,1]B[n-1,0,0,n-2]$. We evaluate $p$ according to Corollary \ref{cor} in two different ways. On one hand we get 
\begin{align*}p &= B[n,0,0,n-1]\left(q^{5-4n}B[n,0,0,n-1]+q^{4-4n}B[n-1,0,0,n-2]\right)\\ 
&=q^{5-4n}B[n,0,0,n-1]^2+q^{4n-17}B[n,0,0,n-1]B[n-2,0,0,n-3]p_1p_0\\ 
&=q^{5-4n}B[n,0,0,n-1]^2\\& \quad +q^{4n-17}\left(q^2B[n-1,0,0,n-2]^2+q^{2n^2-10n+16}p_1^{n}p_0^{n-3}\right)p_1p_0. 
\end{align*}
In the above equations we have used Lemma \ref{plemma} and the induction hypothesis. On the other hand 
\begin{align*}
p&=\left(q^{3-4n}B[n+1,0,0,n]+q^{4-4n}B[n,1,1,n-1]\right)B[n-1,0,0,n-2]\\
&=q^{3-4n}B[n+1,0,0,n]B[n-1,0,0,n-2]+q^{4n-15}B[n-1,0,0,n-2]^2p_1p_0.
\end{align*}
Comparing both expressions for $p$ gives $$B[n+1,0,0,n]B[n-1,0,0,n-2]=q^2B[n,0,0,n-1]^2+q^{2n^2-6n+8}p_1^{n+1}p_0^{n-2}.$$
\end{proof}
Lemma \ref{qexch} is a quantized version of the exchange relation for the cluster algebra.

\subsection{Conclusion}
In the last subsection we draw the conclusion that $U_q^+(s_1s_2s_1s_2)$ carries a quantum cluster algebra structure as defined by Berenstein-Zelevinsky in \cite{BZ2}. 

To be in accord with \cite{BZ2} we rescale our quantized cluster variables. Recall that the generators $u_0$, $u_1$, $u_2$, and $u_3$ of $U_q^+(s_1s_2s_1s_2)$ correspond to the $\Lambda$-modules $T_0$, $T_1$, $T_2$, and $T_3$ of Subsection \ref{rigid}, the dual canonical basis elements $B[n-2,0,0,n-3]$ corresponds to the $\Lambda$-module $T_n$, and $p_0$ and $p_1$ correspond to the $\Lambda$-modules $P_1$ and $P_0$. We rescale each of the above elements by a power of $q$; the exponent is $-\frac{1}{2}$ times the dimension of the endomorphism algebra of the associated $\Lambda$-module. More precisely, introduce elements
\begin{itemize}
 \item $X_0 = q^{-\frac{1}{2}}u_0$, $X_1 = q^{-\frac{1}{2}}u_1$, $X_2 = q^{-\frac{1}{2}}u_2$, $X_3 = q^{-\frac{1}{2}}u_3$,
 \item $Y_0 = q^{-2}p_0$, $Y_1=q^{-2}p_1$,
 \item $X_n =q^{-\frac{1}{2}(2n-5)^2}B[n-2,0,0,n-3]$ for $n \geq 3$
\end{itemize}
in the the algebra $\mathbb{Q}[q^{\pm \frac{1}{2}}] \otimes_{\mathbb{Z}[q,q^{-1}]} U_q^+(w)$ and analogous elements $X_n$ for $n<0$. Here we have enlarged the field of coefficients $\mathbb{Q}(q)$ of $U_q^+(w)$ to contain a square root of $q$. 

Note that in the above examples, the dimension of the endomorphism algebra of the $\Lambda$-module corresponding to $B[a_3,a_2,a_1,a_0]$ is equal to $(a_3+a_2+a_1+a_0)^2$. The exact form of the rescaling exponent was suggested by Leclerc.

For $n \geq 3$ consider four variables $(X_n,X_{n+1},Y_0,Y_1)$ which we group into a cluster. By Lemma \ref{plemma} and Lemma \ref{qcomm} the variables $q$-commute; more precisely, we have $X_nX_{n+1}=q^2X_{n+1}X_n$, $X_nY_0=q^{2n-2}Y_0X_n$, $X_nY_1=q^{-2n+8}Y_1X_n$, $X_{n+1}Y_0 = q^{2n}Y_0X_{n+1}$, $X_{n+1}Y_1=q^{-2n+6}Y_1X_{n+1}$, and $Y_0Y_1=q^{-4}Y_1Y_0$. The matrix $$L=\begin{pmatrix} 0 & 2 & 2n-2 & -2n+8 \\ -2 & 0 & 2n & -2n+6 \\ -2n+2 & -2n & 0 & -4 \\ 2n-8 & 2n-6 & 4 & 0\end{pmatrix}$$ describes the exponents that occur in the commutation relations. The cluster $(X_n,X_{n+1},Y_0,Y_1)$ together with $L$ and the exchange matrix $$B=\begin{pmatrix}0&2 \\ -2&0 \\ n-3 & -n+4 \\ n & -n+1\end{pmatrix}$$ (which is the same as the exchange matrix for the ordinary cluster algebra $\mathcal{A}(\mathcal{C}_M)$) forms a {\it quantum seed} (compare \cite[Definition 4.5]{BZ2}). With every ${\textbf a}=(a_4,a_3,a_2,a_1)\in \mathbb{Z}^4$ Berenstein-Zelevinsky (see \cite[Equation 4.19]{BZ2}) associate an expression
\begin{align*}
M(a_1,a_2,a_3,a_4) &=q^{\frac{1}{2} \sum_{i>j} a_ia_j L_{ij}}X_n^{a_1}X_{n+1}^{a_2}Y_0^{a_3}Y_1^{a_4}\\
&=q^{-\frac{1}{2} \sum_{i>j} a_ia_j L_{ij}}Y_1^{a_4}Y_0^{a_3}X_{n+1}^{a_2}X_n^{a_1}.
\end{align*}
We have $\frac{1}{2} \sum_{i>j} a_ia_j L_{ij} = -a_1a_2-(n-1)a_1a_3+(n-4)a_1a_4-na_2a_3+(n-3)a_2a_4+2a_3a_4$. Lemma \ref{qexch}, written in terms of $X_n,X_{n+1},Y_0,$ and $Y_1$, says that \begin{align*}X_{n+2}X_n&=q^{-2}X_{n+1}^2+q^{-2n^2+6n-3}Y_1^nY_0^{n-3}\end{align*}
Thus, we get an equation for the quantized cluster variable $X_{n+2}$. There holds
\begin{align}X_{n+2}&=q^{-2}X_{n+1}^2X_n^{-1}+q^{-2n^2+6n-3}Y_1^nY_0^{n-3}X_n^{-1} \nonumber \\&= M(-1,2,0,0)+M(-1,0,n-3,n). \label{fin}
\end{align}
Equation (\ref{fin}) is equal to the exchange relation \cite[Equation 4.23]{BZ2} of Berenstein and Zelevinsky. 

The direct sum $\bigoplus_{{\textbf a} \in \mathbb{N}^4}\mathbb{Z}[q^{\pm \frac{1}{2}}]u[{\textbf a}]$ is a $\mathbb{Z}[q,q^{-1}]$-algebra, since the straightening relations involve only polynomials in $q$. It is an integral form of the algebra $\mathbb{Q}[q^{\pm \frac{1}{2}}] \otimes_{\mathbb{Z}[q,q^{-1}]} U_q^+(w)$ defined above. It is generated by $X_0$, $X_1$, $X_2$, and $X_3$ and furthermore contains each $X_n$ for $n \in \mathbb{Z}$. Therefore, it coincides with the $\mathbb{Z}[q,q^{-1}]$-algebra generated by all $X_n$ with $n\in \mathbb{Z}$ which is by definition equal to the quantum cluster algebra as defined in  \cite{BZ2}. 

We conclude with the theorem.
\begin {theorem}
The $\mathbb{Z}[q^{\pm \frac{1}{2}}]$-algebra $$\bigoplus_{{\textbf a} \in \mathbb{N}^4}\mathbb{Z}[q^{\pm \frac{1}{2}}]u[{\textbf a}] \subseteq \mathbb{Q}[q^{\pm \frac{1}{2}}] \otimes_{\mathbb{Z}[q,q^{-1}]} U_q^+(w)$$ is a quantum cluster algebra.
\end {theorem}


\begin{thebibliography}{10}

\bibitem[\textbf{BMRRT} (2006)]{BMRRT}
 A. Buan, R. Marsh, M. Reineke, I. Reiten, G. Todorov,
 \emph{Tilting theory and cluster combinatorics},
 Adv. Math. {\bf 204} (2006), no. 2., 572-618.

\bibitem[\textbf{BZ} (1993)]{BZ}
 A. Berenstein, A. Zelevinsky,
 \emph{String bases for quantum groups of type $A_r$},
 Adv. Soviet Math. {\bf 16} (1993), 51-89.

\bibitem[\textbf{BZ2} (2005)]{BZ2}
 A. Berenstein, A. Zelevinsky,
 \emph{Quantum cluster algebras},
 Adv. Math. {\bf 195} (2005), 405-455.

\bibitem[\textbf{BFZ} (1996)]{BFZ}
 A. Berenstein, S. Fomin, A. Zelevinsky,
 \emph{Parametrizations of canonical bases elements and totally positive matrices},
 Adv. Math. {\bf 122} (1996), 49-149.

\bibitem[\textbf{CC} (2006)]{CC}
 P. Caldero, F. Chapoton,
 \emph{Cluster algebras as Hall algebras of quiver representations},
 Comm. Math. Helv. {\bf 81} (2006), 595-616.

\bibitem[\textbf{CZ} (2006)]{CZ}
 P. Caldero, A. Zelevinsky,
 \emph{Laurent expansions in cluster algebras via quiver representations},
 Moscow Math. J. {\bf 6} (2006), 411-429.

\bibitem[\textbf{CB} (1992)]{CB}
 W. Crawley-Boevey,
 \emph{Lectures on Representations of Quivers}, 
 available at: http://www.amsta.leeds.ac.uk/~pmtwc/quivlecs.pdf, (1992).

\bibitem[\textbf{FG} (2006)]{FG}
 V. Fock, A. Goncharov,
 \emph{Moduli spaces of local systems and higher Teichm\"uller theory},
 Publ. Math. Inst. Hautes \'Etudes Sci. {\bf 103} (2006), 1-211.

\bibitem[\textbf{FZ} (2002)]{FZ}
 S. Fomin, A. Zelevinsky,
 \emph{Cluster algebras I: Foundations},
 J. Amer. Math. Soc. {\bf 15} (2002), 497-529.

\bibitem[\textbf{FZ2} (2003)]{FZ2}
 S. Fomin, A. Zelevinsky,
 \emph{Cluster algebras II: Finite type classification},
 Invent. Math. {\bf 154} (2003), 63-121.

\bibitem[\textbf{FZ3} (2007)]{FZ3}
 S. Fomin, A. Zelevinsky,
 \emph{Cluster algebras IV: Coefficients},
 Compositio Math. {\bf 143} (2007), 112-164.

\bibitem[\textbf{FZ4} (2003)]{FZ4}
 S. Fomin, A. Zelevinsky,
 \emph{Cluster algebras: notes for the CDM-03 conference},
 Current developments in mathematics (2003), 1-34, Int. Press, Somerville, MA, 2003.

\bibitem[\textbf{FZ5} (2003)]{FZ5}
 S. Fomin, A. Zelevinsky,
 \emph{Y -systems and generalized associahedra},
 Math. Ann. {\bf 158} (2003), no. 2, 977-1018.

\bibitem[\textbf{GLS1} (2006)]{GLS1}
  C. Gei\ss , B. Leclerc, Jan Schr\"{o}er,
  \emph{Rigid modules over preprojective algebras},
  Invent. Math. {\bf 165} (2006), 589-632.
 
\bibitem[\textbf{GLS2} (2007)]{GLS2}
  C. Gei\ss , B. Leclerc, Jan Schr\"{o}er,
  \emph{Cluster algebra structures and semicanonical bases for unipotent groups},
  arXiv:math/0703039.
 
\bibitem[\textbf{GSV} (2003)]{GSV}
  M. Gekhtman, M. Shapiro, A. Vainshtein,
  \emph{Cluster algebras and Poisson geometry},
  Mosc. Math. J. {\bf 3} (2003), 899-934.
 
\bibitem[\textbf{H} (1988)]{H}
 D. Happel, 
 \emph{Triangulated categories in the representation theory of finite-dimensional algebras}, 
 LMSLNS {\bf 119 }, Cambridge University Press (1988).

\bibitem[\textbf{KC} (2001)]{KC}
  V. Kac, P. Cheung,
  \emph{Quantum Calculus},
  New York: Springer-Verlag, 2001. 

\bibitem[\textbf{Ka} (1990)]{Ka}
  M. Kashiwara,
  \emph{Bases cristallines},
  C. R. Acad. Sci. Paris S\'er. I Math. {\bf 311} (1990), no. 6, 277-280.

\bibitem[\textbf{K} (2008)]{K}
  B. Keller,
  \emph{Cluster algebras, quiver representations and triangulated categories},
  arXiv:0807.1960.

\bibitem[\textbf{K2} (2005)]{K2}
  B. Keller,
  \emph{On triangulated orbit categories},
  Doc. Math. {\bf 10} (2005), 551-581.
 
\bibitem[\textbf{KR} (2008)]{KR}
  B. Keller, I. Reiten,
  \emph{Acyclic Calabi-Yau categories},
  Compos. Math.  {\bf 144}  (2008),  no. 5, 1332-1348. 

\bibitem[\textbf{La} (2007)]{La}
  P. Lampe,
  \emph{Cluster Algebren vom Rang 2},
  diploma thesis, available at: http://www.math.uni-bonn.de/people/lampe/Diplomarbeit.pdf.
 
\bibitem[\textbf{Le1} (2008)]{Le1}
  B. Leclerc,
  \emph{A canonical basis in type $A^{(1)}_1$},
  private communication.

\bibitem[\textbf{Le2} (2004)]{Le2}
  B. Leclerc,
  \emph{Dual canonical bases, quantum shuffles and q-characters},
  Math. Z. {\bf 246} (2004), 691-732.

\bibitem[\textbf{Le3} (2003)]{Le3}
  B. Leclerc,
  \emph{Imaginary vectors in the dual canonical basis of $U_q(\mathfrak{n})$},
  Transform. Groups {\bf 8} (2003), 95-104.

\bibitem[\textbf{Le4} (2009)]{Le4}
  B. Leclerc,
  \emph{Canonical and semicanonical bases},
  talk at the University of Reims, available at: http://loic.foissy.free.fr/colloque/Leclerc.pdf.

\bibitem[\textbf{LNT} (2003)]{lnt}
  B. Leclerc, M. Nazarov, J.-Y. Thibon,
  \emph{Induced representations of anne Hecke algebras and canonical bases of quantum groups},
  Progr. Math. {\bf 210} (2003) 115-153.

\bibitem[\textbf{Lu1} (1993)]{Lu1}
  G. Lusztig,
  \emph{Introduction to quantum groups},
  Birkh\"{a}user Prog. Math.
  {\bf 110,} 
  (1993).

\bibitem[\textbf{Lu2} (1990)]{Lu2}
  G. Lusztig,
 \emph{Canonical bases arising from quantized enveloping algebras},
 J. Amer. Math. Soc. {\bf 3} (1990), 447-498.

\bibitem[\textbf{MP} (2007)]{MP}
  G. Musiker, J. Propp,
  \emph{Combinatorial interpretations for rank-two cluster algebras of affine type},
  Electron. J. Combin. {\bf 14} (2007), R15.

\bibitem[\textbf{Re} (1999)]{Re}
 M. Reineke,
 \emph{Multiplicative Properties of Dual Canonical Bases of Quantum Groups},
 J. Algebra {\bf 211} (1999), 134-149.

\bibitem[\textbf{Ri} (1998)]{R}
 C. M. Ringel,
 \emph{The preprojective algebra of a quiver},
 Algebras and modules, II (Geiranger), 467-480, CMS Conf. Proc. {\bf 24} (1998), Amer. Math. Soc., Providence, RI.

\bibitem[\textbf{Sk} (2007)]{Sk}
  M. Skandera,
  \emph{The cluster basis of $\mathbb{Z}[x_{1,1},\ldots,x_{3,3}]$},
  Electron. J. Comb. {\bf 14} (2007), \#R76.

\bibitem[\textbf{Sz} (2009)]{Sz}
  C. Sz\'ant\'o,
  \emph{On the cardinalities of Kronecker quiver Grassmannians},
  arXiv:0903.1928v2.

\bibitem[\textbf{SZ} (2004)]{SZ}
  P. Sherman, A. Zelevinsky,
  \emph{Positivity and canonical bases in rank two cluster algebras of finite and affine types},
  Moscow Math. J. {\bf 4} (2004), 947-974.

\bibitem[\textbf{Z} (2006)]{Z}
  A. Zelevinsky,
  \emph{Semicanonical basis generators of the cluster algebra of type $A_1^{(1)}$},
  arXiv:math/0606775.

\end{thebibliography}
\end{document}